\documentclass[11pt,a4paper,final]{amsart}
\usepackage[T1]{fontenc}
\usepackage[UKenglish]{babel}
\usepackage[a4paper,margin=1in]{geometry}

\usepackage{amsmath,amsthm,amsfonts,amssymb}
\usepackage{mathrsfs,dsfont}
\usepackage{graphicx}
\usepackage{url}
\usepackage{verbatim}
\usepackage{xcolor}

\usepackage{marginnote}

\usepackage{hyperref}

\newcommand{\normal}{\color{black}}

\makeatletter
\def\namedlabel#1#2{\begingroup
    #2%
    \def\@currentlabel{#2}%
    \label{#1}\endgroup
}
\makeatother
\theoremstyle{plain}
\newtheorem{theorem}{Theorem}[section]
\newtheorem{corollary}[theorem]{Corollary}
\newtheorem{lemma}[theorem]{Lemma}
\newtheorem{proposition}[theorem]{Proposition}

\theoremstyle{definition}
\newtheorem{remark}[theorem]{Remark}

\newtheorem{definition}[theorem]{Definition}

\numberwithin{equation}{section}

\renewcommand\labelenumi{\textup{\alph{enumi})}}

\renewcommand\theenumi\labelenumi%{\textup{\alph{enumi})}}
\makeatletter\renewcommand{\p@enumii}{}\makeatother %\ref concatenates the \theenumi's with
\renewcommand{\leq}{\leqslant}

\renewcommand{\geq}{\geqslant}

\newcommand{\cA}{\mathcal{A}}

\newcommand{\R}{\mathds{R}}

\newcommand{\Z}{\mathds{Z}}
\newcommand{\pr}{\mathbf{P}}
\newcommand{\qpr}{\mathds{Q}}

\newcommand{\ex}{\mathbf{E}}

\newcommand{\Dom}{\mathop{\mathrm{Dom}}}

\begin{document}
\title[Lifschitz tail for alloy-type model]
{Lifschitz tail for alloy-type models driven by the fractional Laplacian}

\author[K.~Kaleta]{Kamil Kaleta}
\address[K.~Kaleta]{Faculty of Pure and Applied Mathematics\\ Wroc{\l}aw University of Science and Technology\\ ul. Wybrze{\.z}e Wyspia{\'n}skiego 27, 50-370 Wroc{\l}aw, Poland}
\email{kamil.kaleta@pwr.edu.pl}

\thanks{Research supported by the National Science Center, Poland, grant no.\ 2015/17/B/ST1/01233.}

\author[K.~Pietruska-Pa{\l}uba]{Katarzyna Pietruska-Pa{\l}uba}
\address[K.~Pietruska-Pa{\l}uba]{Institute of Mathematics \\ University of Warsaw
\\ ul. Banacha 2, 02-097 Warszawa, Poland}
\email{kpp@mimuw.edu.pl}

\maketitle

\baselineskip 0.5 cm

\begin{abstract}
We establish precise asymptotics near zero of the integrated density of states for the random Schr\"{o}dinger operators $(-\Delta)^{\alpha/2} + V^{\omega}$ in $L^2(\mathbb R^d)$ for the full range of $\alpha\in(0,2]$ and a fairly large class of random nonnegative alloy-type potentials $V^{\omega}$. The IDS exhibits the Lifschitz tail singularity. We prove the existence of the limit $$\lim_{s\to 0} s^{d/\alpha}\ln\ell([0,s]) = -C \left(\lambda_d^{(\alpha)}\right)^{d/\alpha},$$
with $C \in (0,\infty]$. The constant $C$ is  is finite if and only if the common distribution of the lattice random variables charges $\left\{0\right\}$. In this case, the constant $C$ is expressed explicitly in terms of such a probability. In the limit formula, $\lambda_d^{(\alpha)}$ denotes the Dirichlet ground-state eigenvalue of the operator $(-\Delta)^{\alpha/2}$ in the unit ball in $\mathbb R^d.$

\smallskip


\smallskip

\end{abstract}

\bigskip\bigskip
\section{Introduction}

In this paper we study the random  Schr\"{o}dinger  operator in the $\mathbb R^d-$setting:
\begin{equation}\label{eq:oper-def}
H^\omega=(-\Delta)^{\alpha/2} +V^\omega, \quad \alpha \in (0,2],
\end{equation}
with the potential
\begin{equation}\label{eq:pot-def}
V^\omega (x) = \sum_{\mathbf i\in\Z^d} q_{\mathbf i}(\omega) W(x-\mathbf i),\quad x\in\mathbb R^d,
\end{equation}
where $\{q_{\mathbf i}\}_{\mathbf i\in \Z^d}$ is a sequence of i.i.d.\ nonnegative  and nondegenerate  random variables  over the probability space $(\Omega, \mathcal A, \mathbb Q),$ with cumulative distribution function $F_q(t)=\mathbb Q[q\leq t],$ and $W:\R^d\to\mathbb [0,\infty)$ is a sufficiently regular  nonnegative  \emph{single-site potential}. Such a potential is said to be {\em alloy-type}. For $\alpha \in (0,2)$ the fractional Laplacian operator $-(-\Delta)^{\alpha/2}$ is a non-local operator, and for $\alpha=2$ it becomes the usual Laplacian $\Delta$. We are mainly interested in the study of the asymptotic behavior of {\em the integrated density of states} (IDS) for the operator $H^\omega$ (denoted by $\ell$) at the bottom of its spectrum (the precise definition of $\ell$ is given in Section \ref{sec:existence}).

In the discrete setting - when the Schr\"{o}dinger operator is based on the   discrete
Laplacian on $\ell^2(\mathbb Z^d) -$  such operators were widely studied and the literature of the subject is immense. They are sometimes called Anderson operators and the evolution based on $H^\omega$ -- the Parabolic Anderson Model (PAM). The reader
interested in discrete rather than continuous models may consult,
e.g. the books  \cite{bib:Car-Lac}, \cite{bib:Mol}, \cite{bib:Fig-Pas}, and the more recent monographs \cite{bib:Kon}, \cite{bib:Aiz-War},  together with the literature therein.
 For the particular case of alloy-type potentials, for a survey of results we also refer to the review paper \cite{bib:Elg-Kru-Tau-Ves}.

 The literature of the case of the Laplacian on $L^2(\mathbb R^d)$ disturbed by a random potential is, by now,  also widespread. Of the singular-type potentials, the best-analysed case is that of Poisson potentials:
\begin{align} \label{eq:Poiss}
 V^\omega= \int_{\mathbb R^d} W(x-y)\mu^\omega({\rm d}y),
\end{align}
 where $\mu^\omega$ is a random Poisson point measure in $\mathbb R^d.$  These operators are known to have Lifschitz tails,
 see e.g. \cite{bib:Oku,bib:Szn1,bib:Szn-book}
 which is a strong indication for the localization property in that case, proven later in \cite{bib:Ger-His-Kle}.
 Alloy-type potentials in the continuum were considered e.g. in the celebrated paper by
 Bourgain and Kenig \cite{bib:Bou-Ken}, and earlier also in \cite{bib:Com-His}, \cite{bib:Dam-Sim-Sto}.
As to the Lifschitz tail itself, it has been proven by Kirsch and Simon  in the continuous alloy-type model in \cite{bib:Kir-Sim}. This paper is also the starting point for our considerations  and it motivates the main questions we address in our present contribution.

In the paper \cite{bib:Kir-Sim}, the authors have considered the Schr\"odinger operator $H^\omega=-\Delta +V^\omega$ with $V^\omega$ being a sum of the random lattice potential as in \eqref{eq:pot-def} and a sufficiently regular $\mathbb Z^d-$periodic potential. The assumptions of that paper are as follows:
the support of the distribution of the $q$'s is a compact subset of the positive half-line (but not a single point) and is touches zero, their common distribution function satisfies $F^-(t)\geq Ct^N$ for some constants  $C,N>0,$ and $t$ close to zero,
the single-site potential $W$ is a function satisfying $W(x)=O(|x|^{-d-\epsilon})$ as $|x|\to\infty$ and some further technical integrability conditions.

Under these assumptions, the IDS exists, and the authors of the cited paper prove the Lifschitz tail: when $\epsilon\geq 2,$ then
 at the bottom of the spectrum $\lambda_0,$
\[\lim_{\lambda\searrow \lambda_0} \frac{\ln\{-\ln \ell([0,\lambda]\}}{\ln( \lambda-\lambda_0)}=-\frac{d}{2}.\]
This strong result motivates further important questions about the asymptotic behavior of $\ell([0,\lambda]$ as $\lambda \searrow 0$ and its actual dependence on
the initial data provided by the $q$'s and the single site potential $W$.
 Let us note that the authors of \cite{bib:Kir-Sim} were not able to obtain the existence of the limit
\begin{equation}\label{eq:ks-fail}
\lim_{\lambda\searrow \lambda_0} (\lambda-\lambda_0)^{d/2}\ln \ell([0,\lambda]);
\end{equation} in the proof there were lower order terms distorting the picture. It was even not clear if for the alloy-type models this limit could exists at all.
Interestingly, the existence of the finite limit as in \eqref{eq:ks-fail} would mean that the IDS of this particular alloy-type model manifests the behaviour known from
the models based on the Poissonian type random fields as in \eqref{eq:Poiss} (see e.g. \cite{bib:Oku,bib:Szn1}). Note that these two types of potentials induce two different models of random environments which typically require completely different approaches and methods. All the above questions are addressed in the present paper.

In our framework, we replace the kinetic term $-\Delta$ by the more general operators $(-\Delta)^{\alpha/2},$ for the full range of $\alpha\in(0,2]$. To the best of our knowledge, the use of nonlocal operators for alloy-type models is a novelty. We show
that once $W$ is a bounded, compactly supported and nontrivial single site potential, and the support of the distribution of $q$'s touches zero, then the limit in \eqref{eq:ks-fail} exists. Moreover, we prove that in these settings it is finite if and only if the distribution of $q$'s charges zero, i.e. $F_q(0)= \mathbb Q[q=0]>0$. All of our framework assumptions on $W$ and $q$'s are precisely stated in {\bf(W1)}-{\bf(W2)} and {\bf(Q)} at the beginning of Section \ref{sec:random_set_up}.

 Our main result is as follows.

 {
\begin{theorem}\label{coro-no-atom-at-0_IDS}
Let $\alpha \in (0,2]$ and let $H^{\omega}$ be the Schr\"odinger operator with the lattice potential $V^{\omega}$ as in \eqref{eq:oper-def}-\eqref{eq:pot-def} such that the assumptions {\bf(W1)}-{\bf(W2)} and {\bf(Q)} hold.
We have the following statements.
     \begin{itemize}
        \item[(1)] If the distribution of $q$ has an atom at zero, i.e.\ $ F_q(0)>0,$ then
	                 \[\lim_{\lambda \searrow 0} \lambda^{d/\alpha} \ln \ell[0,\lambda]=	 - \ln\left(\frac{1}{ F_q(0)}\right) \left(\lambda_d^{(\alpha)}\right)^{d/\alpha},\]
									where $\lambda_d^{(\alpha)}$ is the ground state eigenvalue of the operator $(-\Delta)^{\alpha/2}$ constrained to the unit ball in $\R^d$, i.e.\ with Dirichlet conditions outside of this ball \textup{(}for $\alpha \in (0,2)$\textup{)}, and on its boundary \textup{(}for $\alpha=2$\textup{)}.
									
	      \item[(2)] If the distribution of $q$ has no atom at zero, i.e.\ $F_q(0)=0$, then
	                 \[\lim_{\lambda \searrow 0} \lambda^{d/\alpha} \ln \ell[0,\lambda]=	 -\infty.\]
     \end{itemize}
\end{theorem}
	}
Before we discuss the set-up for our research, it is useful to give some interpretations of the above result. It shows that in general the two scenarios are possible.
If $F_q(0)>0$ (i.e.\ $q$ takes the value $0$ with positive probability $F_q(0)$), then the lattice system manifests the behavior known from the Poissonian settings -- this is the case (1). The constant $\ln\left(\frac{1}{ F_q(0)}\right)$ is a lattice counterpart of the intensity of the Poisson cloud of points in the Poisson model (cf. e.g. \cite[eq. (1.1)]{bib:Nak} for $\alpha=2$ and \cite[eq (1.8)]{bib:Oku} for $\alpha \in (0,2]$). Observe that our result in (1) says that
$$
\ell[0,\lambda] =	 F_q(0)^{\left|B(r_{\alpha}(\lambda))\right|(1+o(1))}, \quad \text{as} \ \ \lambda \searrow 0,
$$
where $r_{\alpha}(\lambda):= (\lambda_d^{(\alpha)}/\lambda)^{1/\alpha}$.
Following Sznitman (cf. \cite[Remark\ 3.6(1)]{bib:Szn1}), it can be interpreted as follows. Since $F_q(0)$ is the probability of the $q_{\mathbf i}$ being zero at any given lattice point $\mathbf i,$  $\ell[0,\lambda]$
behaves roughly as the probability that  in the ball with ground state eigenvalue equal  to $\lambda,$ all the $q_{\mathbf i}$'s are zero -- meaning that the potential comes only from lattice point outside of this ball.

 On the other hand, when $F_q(0)=0$, then (2) shows that the behaviour of $\ell[0,\lambda]$ is different, placing the system in another scenario. In this regime the actual decay properties of $F_q(t)$ as $t \searrow 0$ affect the asymptotic behaviour of IDS at zero, giving  corrections to the rate $\lambda^{d/\alpha}$. This regime is analyzed in detail in a companion paper \cite{bib:KK-KPP-alloy2}, in which we study this problem for a substantially wider class of non-local random Schr\"odinger operators and more general single site potentials $W$. This further research was strongly motivated by the dichotomy that we obtained in our Theorem \ref{coro-no-atom-at-0_IDS} presented above.

Let us now say a few words about our proofs. Our methods are mainly probabilistic and they rely on an application of the Feynman-Kac
representations of the evolution (heat) semigroups considered. At the probabilistic side, the Laplacian is the generator of the Brownian motion. Since we replace this operator by  more general operators $-(-\Delta)^{\alpha/2},$ $\alpha\in(0,2]$, we need to study the full range of the isotropic $\alpha$-stable processes which give rise to these operators. For $\alpha \in (0,2)$ such processes are pure jump L\'evy processes, and the Brownian motion is the only isotropic stable process with continuous paths (local vs. nonlocal nature of their generators).

In our approach, we work with Laplace transform $L$ of $\ell$ and prove the long-time asymptotics of $L,$ which we then transform to the Lifschitz tail asymptotics by means of an exponential-type Tauberian theorem of \cite{bib:F} (Section \ref{sec:asymp}). Such an approach was successfully used for
the  Poissonian potentials: for the  Brownian motion  in $\mathbb R^d$
\cite{bib:Szn1} and on fractals \cite{bib:KPP1}, for L\'{e}vy processes in $\mathbb R^d$ \cite{bib:Oku}, for subordinate Brownian
motions on fractals \cite{bib:KK-KPP2}. The Laplace transform we study is closely related to the trace of Feynman-Kac semigroups, which are
expressed and analyzed mostly probabilistically. The lower bound (Theorem \ref{th:dolne}) is proven directly. The proof of the upper bound  (Theorem \ref{th:upper-stable}) is the most demanding part of the paper. It consists of several key steps. For translation-invariant lattice potentials,  for the Laplace transform of the IDS is given explicitly as the Feynman-Kac integral of the
$\alpha-$stable process with the potential given by \eqref{eq:pot-def}. As the first step, we periodize the lattice random variables $q.$ - with given $M>0,$ we repeat the realization of the $q_{\mathbf i}$'s from $[0,M)^d$ over entire $\mathbb Z^d.$ Next, we replace the $\alpha-$stable process on $\mathbb R^d$ with the process
on the torus of size $M,$ denoted $\mathcal T_M,$  and the random variables $q_{\mathbf i}$ with certain Bernoulli variables. All these operations can only increase the Feynman-Kac functional. In the next step we make use of the scaling properties of $\alpha-$stable processes -- we consider tori of size $M=mK$ and shrink the state-space with $K.$ The resulting functional represents the trace of the semigroup of the stable process on $\mathcal T_m,$ affected by the appropriately rescaled potential. The decay of the trace is asymptotically governed by the
principal eigenvalue of this semigroup.
To estimate this eigenvalue, we employ the coarse-graining  method  in the form of `enlargement of obstacles' from \cite{bib:Szn1} or rather its nonlocal version of \cite{bib:KK-KPP2} which permits to manage the possibly intricate configuration of the lattice random variables in the torus $\mathcal T_m.$  As the last step, we identify the constant resulting in the limit. Let us note that we were able to use, in the lattice case, the Sznitman's method, devised to work for Poisson random potentials. This is a new approach, even for the Brownian motion.

The paper is organized as follows. In Section \ref{sec:prel} we present the basic facts on $\alpha-$stable processes and the corresponding semogroups, together with a short discussion of random Schr\"{o}dinger operators. In Section \ref{sec:existence} we define the integrated density and  give an explicit formula  which will be important in further sections. Section \ref{sec:asymp} is devoted to a detailed discussion of our main results:
the upper bound in Section \ref{sec:upper} and then the matching lower bound in Section \ref{sec:lower}. Finally, in Sections \ref{sec:upper-proof} and \ref{sec:lower-proof} we give the proofs.

\section{Preliminaries}\label{sec:prel}
\subsection{The isotropic stable process}\label{sec:process}

Our approach in this paper is in large part probabilistic -- we study the evolution semigroups of the operators $H^{\omega}$ and $H^{\omega}_\Lambda$ through the Feynman--Kac formula. Therefore we first give the probabilistic background. Let $(Z_t)_{t \geq 0}$ be the isotropic $\alpha-$stable process on $\mathbb R^d$, $\alpha \in (0,2]$, $d \geq 1$, i.e. the L\'evy process determined by $\ex e^{i \xi \cdot Z_t} = e^{-t|\xi|^{\alpha}}$, $\xi \in \R^d$, $t>0$ \cite{bib:Sat}. For $\alpha = 2$ the process $(Z_t)_{t \geq 0}$ is the standard Brownian motion running at double speed, i.e. the diffusion process such that $\pr (Z_t \in {\rm d}z) = p(t,z) {\rm d}z$ with $p(t,z)=(4\pi t)^{-d/2} e^{-|z|^2/(4t)}$ being the Gauss kernel.
For $\alpha \in (0,2)$ it is a pure jump L\'evy process with c\`adl\`ag paths whose L\'evy--Khintchine exponent is given by
$$
|\xi|^{\alpha} = \int_{\R^d \setminus \left\{0\right\}} \big(1-\cos( \xi \cdot z)\big) \nu({\rm d}z),
$$
where $\nu({\rm d}z) = \nu(z){\rm d}z = \normal \cA_{d,-\alpha} |z|^{-d-\alpha} {\rm d}z$, $\cA_{d,\gamma} = \Gamma((d-\gamma)/2)/(2^{\gamma}\pi^{d/2}|\Gamma(\gamma/2)|)$, is the L\'evy measure (the jump intensity).  In the cases when we need to specify the `start' and `end' points of the jumps, we will also use the notation $\nu(x,y)=\nu(x-y).$ In the pure jump case, we also have $\pr(Z_t \in {\rm d}z) = p(t,z){\rm d}z$; the density $p(t,z)$ is a continuous function on $(0,\infty) \times \R^d$ which enjoys the two-sided sharp estimates $$p(t,z) \asymp t^{-d/\alpha} \wedge t|z|^{-d-\alpha}, \quad z \in \R^d, \quad t>0.$$ In either case, $(Z_t)_{t \geq 0}$ is a strong Feller process having the scaling property
$\pr(Z_t \in {\rm d}z) = \pr(rZ_{r^{-\alpha} t} \in {\rm d}z)$. Throughout the paper, by $\pr_x$ we denote the probability measure for the process starting from $x \in \R^d$, i.e. $\pr_x(Z_t \in {\rm d}z) = \pr(Z_t+x \in {\rm d}z)$, and by $\ex_x$, the corresponding expected value.  In particular, $\pr_x(Z_t \in {\rm d}z) = p(t,x,z){\rm d}z$, where $p(t,x,z):= p(t,z-x)$. Then the operator $-(-\Delta)^{\alpha/2}$ is the infinitesimal generator of the process $(Z_t)_{t \geq 0},$ see e.g. \cite{bib:Kwa}.

Below we will also use the bridge measure of the stable process starting from $x$ and conditioned to have $Z_t=y,$ $\mathbf P_x-$almost surely.
More precisely, for fixed $t>0$ and $x,y \in \R^d,$ the bridge measure $\mathbf P_{x,y}^{t}$ is defined by the following property:
for any $0 <s<t$ and $A\in \sigma(Z_u: u \leq s),$
\begin{equation}\label{eq:bridge}
\pr_{x,y}^{t}[A] = \frac{1}{p(t,x,y)}\mathbf E_x[\mathbf 1_A p(t-s,Z_s,y)]
\end{equation}
and then extended to $s=t$ by weak continuity.
For more detailed information on Markovian bridges we refer to \cite{bib:Cha-Uri}.

\subsection{Random Schr\"{o}dinger operators and semigroups} \label{sec:random_set_up}

Let $V^{\omega}$ be the random potential defined in \eqref {eq:pot-def}, with the single-site potential $W:\mathbb R^d\to\mathbb [0,\infty)$ (also called the \emph{potential profile}). Throughout the paper we assume that $W$ satisfies the following two conditions:
\begin{itemize}
\item[\bf (W1)] $W$ is bounded and of bounded support: there exists $a>0$ such that $\mbox{supp}\, W\subset B(0,a);$
\item[\bf (W2)] there exist $0<a_0\leq a$ and $b>0$ such that  $W(x)\geq b$ for $x\in B(0, a_0).$
\end{itemize}

\noindent
We also assume that
\begin{itemize}
\item[\bf (Q)] $q_{\mathbf i}$ are nonnegative, nondegenerate, i.i.d. random variables, and $F_q(\kappa)>0$ for any $\kappa>0.$
\end{itemize}

It follows from {\bf (W1)} that the potential $V^{\omega}$ is well-defined,  nonnegative and bounded. This allows us to define the Schr\"odinger operator $H^\omega=(-\Delta)^{\alpha/2} +V^\omega$ as a self-adjoint, positive operator on $\Dom(H^{\omega}) = \Dom\big((-\Delta)^{\alpha/2}\big):= \left\{f \in L^2(\R^d,{\rm d}\xi): |\xi|^{\alpha} \widehat f \in L^2(\R^d,{\rm d}\xi)\right\}$. In particular, the evolution semigroup of the operator $H^{\omega}$ has a probabilistic representation with respect to the process $(Z_t)_{t \geq 0},$ given by the Feynman--Kac formula:
$$
e^{-tH^{\omega}} f(x) = T_t^{V^{\omega}} f(x):= \ex_x\left[{\rm e}^{-\int_0^t V^{\omega}(Z_s) {\rm d}s} f(Z_t)\right], \quad f \in L^2(\R^d, {\rm d}x), \quad t>0.
$$
 Our standard reference for Schr\"odinger operators based on generators of L\'evy processes is the monograph \cite{bib:DC} by Demuth and van Casteren.  Denote by $H^{\omega}_\Lambda$  the operator $H^{\omega}$ constrained to a bounded, nonempty region $\Lambda\subset \R^d$ (we consider Dirichlet conditions on $\Lambda^c$ when $\alpha\in(0,2)$ and on $\partial \Lambda$ when $\alpha=2$) and let ${\rm e}^{-tH^{\omega}_{\Lambda}}, t\geq 0,$ be its evolution (heat) semigroup on $L^2(\Lambda, {\rm d}x)$. Similarly as above, we have:
$$
{\rm e}^{-tH^{\omega}_{\Lambda}} = T_t^{V^{\omega},\Lambda} f(x):= \ex_x\left[{\rm e}^{-\int_0^t V^{\omega}(Z_s) {\rm d}s} f(Z_t); t<\tau_{\Lambda}\right], \quad f \in L^2(\Lambda, {\rm d}x), \quad t>0.
$$
Here $\tau_{\Lambda}:=\inf\{t\geq 0: \, Z_t\notin \Lambda\}$ denotes the first exit time of the process from the domain $\Lambda$. The semigroup operators $T_t^{V^{\omega},\Lambda}$, $t>0$, are integral operators, i.e. there exist measurable, symmetric and bounded kernels $p^{V^{\omega}}_{\Lambda}(t,x,y)$ such that
\begin{align}
\label{def:sem-dir-kernel}
T_t^{V^{\omega},\Lambda} f(x) = \int_{\Lambda} p^{V^{\omega}}_{\Lambda}(t,x,y) f(y){\rm d}y, \quad f \in L^2(\Lambda,{\rm d}x),  \quad t>0;
\end{align}
these kernels are given by the formula
\begin{align}
\label{def:sem-dir-kernel-bridge}
p^{V^{\omega}}_{\Lambda}(t,x,y) = p(t,x,y) \ \ex^t_{x,y}\left[{\rm e}^{-\int_0^t V^{\omega}(Z_s) {\rm d}s} ; t<\tau_{\Lambda}\right],\quad x,y \in \Lambda, \quad t>0,
\end{align}
where $p(t,x,y)$ is the transition density of the isotropic $\alpha-$stable process, and $\ex_{x,y}^t$ is the expectation with respect to the stable bridge measure $\pr_{x,y}^t$ .

Since $|\Lambda|<\infty$ and the kernels $p^{V^{\omega}}_{\Lambda}(t,x,y)$ are bounded, all  the operators $T_t^{V^{\omega},\Lambda}$, $t>0$, are Hilbert-Schmidt. In particular, there exists a complete orthonormal set $\left\{\varphi_k^{V^{\omega},\Lambda}\right\}_{k=1}^{\infty}$ in $L^2(\Lambda,{\rm d}x)$, consisting of eigenfunctions of the operator $H^{\omega}_{\Lambda}$. The corresponding eigenvalues satisfy $0\leq \lambda_1^{V^{\omega}}(\Lambda) < \lambda_2^{V^{\omega}}(\Lambda)\leq \lambda_3^{V^{\omega}}(\Lambda) \leq \ldots \to \infty$; each $\lambda_k^{V^{\omega}}(\Lambda)$ is of finite multiplicity and the ground state eigenvalue $\lambda_1^{V^{\omega}}(\Lambda)$ is simple.

When $V^{\omega} \equiv 0,$ then we just simply write $T_t^{\Lambda}$, $p_{\Lambda}(t,x,y)$ and
$\lambda_k(\Lambda)$ etc.

\section{Existence of the density of states} \label{sec:existence}
As the potential \eqref{eq:pot-def} is stationary with respect to
$\Z^d,$ the existence of the density of states follows from
general theory.
 More precisely,
for a given domain $\Lambda\subset \R^d$ such that $0<|\Lambda|<\infty,$    let
\[\ell_{\Lambda}^{\omega}(\cdot)=\frac{1}{|\Lambda|}\sum_{k=1}^\infty\delta_{\lambda_k^{V^{\omega}}(\Lambda)}
(\cdot)\]
be the counting measure on the spectrum of $H_{\Lambda}^{\omega},$ normalized by the volume.
Due to stationarity properties of $V^{\omega},$ we restrict our attention to sets $\Lambda$ composed of  unit cubes with  vertices in $\Z^d.$
 From the maximal ergodic theorem
(see \cite[Remark VI.1.2]{bib:Car-Lac}) we get that the measures
$\ell_{\Lambda}^{\omega}$ converge vaguely, as $\Lambda \nearrow \R^d,$ to a nonrandom measure $\ell,$ which is
the density of states of $H^\omega.$ The  vague convergence of $\ell_{\Lambda}^{\omega}$ when $\Lambda \nearrow \R^d$ amounts to the convergence of their Laplace transforms
\begin{eqnarray*}
L_\Lambda^\omega(t)&=&\frac{1}{|\Lambda|} \int_{[0,\infty)}{\rm e}^{-t\lambda}\ell_{\Lambda}^{\omega}({\rm d}\lambda)= \frac{1}{|\Lambda|}{\rm Tr}\, T_t^{V^{\omega},\Lambda}
=
 \frac{1}{|\Lambda|}\int_\Lambda p^{V^{\omega}}_{\Lambda}(t,x,x) {\rm d}x
\\
&=&
 \frac{p(t,0)}{|\Lambda|}\int_\Lambda \ex_{x,x}^t\left[{\rm e}^{-\int_0^t V^\omega(Z_s){\rm d}s}; t < \tau_\Lambda \right] {\rm d}x
 \end{eqnarray*}
 for any fixed $t>0.$
 Moreover, we have that the limit $L(t)$ is the Laplace transform of
 $\ell,$ and for any $t>0$
\begin{equation}\label{eq:el-almost-everywhere}
L(t) = \lim_{\Lambda \nearrow \R^d} \mathbb E^{\mathbb Q}L_\Lambda^\omega(t).
\end{equation}

We can actually write down an expression for $L.$
\begin{proposition} Let $L(t)$ be the Laplace transform of the integrated density of states, $\ell.$ Then
\begin{equation}\label{eq:el-expression}
L(t)= p(t,0,0)\int_{[0,1)^d} \mathbb E^{\mathbb Q}\mathbf E_{x,x}^t\left[{\rm e}^{-\int_0^t V^\omega(X_s){\rm d}s}\right] {\rm d}x.
\end{equation}
\end{proposition}

\begin{proof}
This formula follows from the stationarity of the potential $V^{\omega}$ with respect to $\Z^d,$ and can be deduced from
\eqref{eq:el-almost-everywhere}. Specializing to $\Lambda_m=[0,m)^d,$ we write
\[L(t)=\lim_{m\to\infty}\mathbb E^{\mathbb Q} L_{\Lambda_m}^\omega(t).\]
Clearly,
\begin{eqnarray*}
\mathbb E^{\mathbb Q}L_{\Lambda_m}^\omega(t)
&=& \frac{ p(t,0,0) }{|\Lambda_m|} \int_{\Lambda_m}\mathbb E^{\mathbb Q}
\mathbf E_{x,x}^t\left[{\rm e}^{-\int_0^t V^\omega(X_s){\rm d}s}\right] {\rm d}x\\
&&  - \ \frac{ p(t,0,0) }{|\Lambda_m|} \int_{\Lambda_m}\mathbb E^{\mathbb Q}
\mathbf E_{x,x}^t\left[{\rm e}^{-\int_0^t V^\omega(X_s){\rm d}s};\tau_{\Lambda_m} \leq t\right] {\rm d}x.
\end{eqnarray*}
Since for any unit cube $C$ in $\R^d$ with vertices in $\Z^d$  the expression
\[\int_C \mathbb E^{\mathbb Q}\mathbf E_{x,x}^t\left[{\rm e}^{-\int_0^t V^\omega(X_s){\rm d}s}\right] {\rm d}x
\]
does not depend on $C,$
we get
\[
L(t)\leftarrow \mathbb E^{\mathbb Q}L_{\Lambda_m}^\omega(t)=  p(t,0,0) \int_{[0,1)^d} \mathbb E^{\mathbb Q}\mathbf E_{x,x}^t\left[{\rm e}^{-\int_0^t V^\omega(Z_s){\rm d}s}\right] {\rm d}x - E_m , \quad  m \to \infty,
\]
where
\[E_m=\frac{ p(t,0,0) }{|\Lambda_m|} \int_{\Lambda_m}\mathbb E^{\mathbb Q}
\ex_{x,x}^t\left[{\rm e}^{-\int_0^t V^\omega(Z_s){\rm d}s};\tau_{\Lambda_m} \leq t\right] {\rm d}x\]
 is the error term.
We only need to show that $E_m\to 0$ as $m \to \infty.$ The proof goes as follows:
\begin{eqnarray*}
|E_m|&\leq & \frac{p(t,0,0)}{|\Lambda_m|}\int_{\Lambda_m} \pr_{x,x}^t\left[\tau_{\Lambda_m} \leq t\right] {\rm d}x
\\
&\leq& p(t,0,0)\left[\frac{1}{m^d}\int_{[\sqrt m, m-\sqrt m)^d} \mathbf \pr_{x,x}^t[\sup_{s\leq t}|Z_s-x|\geq \sqrt m]\,{\rm d}x
+ \frac{m^d- (m-2\sqrt m)^d}{m^d}\right]\\
&\leq & p(t,0,0)\left[\mathbf \pr_{0,0}^t[\sup_{s\leq t}|Z_s|\geq\sqrt m]+  \frac{m^d- (m-2\sqrt m)^d}{m^d}\right]\to 0 \quad \mbox{ as }\,\,  m\to \infty.
\end{eqnarray*}
\end{proof}

\noindent As an immediate consequence of \eqref{eq:el-expression}
we obtain the following two formulas.
\begin{corollary}\label{coro-el-t-properties}
(i) Let $\Lambda\subset \R^d$ be a set composed with unit cubes with vertices in $\Z^d.$
Then
\[L(t)=\frac{ p(t,0,0) }{|\Lambda|}\int_\Lambda \mathbb E^{\mathbb Q}\mathbf E_{x,x}^t\left[{\rm e}^{-\int_0^t V^\omega(Z_s){\rm d}s}\right] {\rm d}x.
\]
(ii)
Let $B_r=B(0,r)\subset \R^d$ be a ball.
Then for any $t>0,$
      \[ L(t)= (1+o(1))\frac{ p(t,0,0) }{|B_r|}\int_{B_r} \mathbb E^{\mathbb Q} \mathbf E_{x,x}^t\left[{\rm e}^{-\int_0^t V^\omega(Z_s){\rm d}s}   \right]\,{\rm d}x ,\quad r\to\infty\]
 (the error term $o(1)$ does not depend on $t$).
\end{corollary}
\begin{proof}
Part (i) is obvious. To prove (ii), for a given ball $B_r$ let $\Lambda_r^0\subset B_r\subset \Lambda_r^1$ be two sets composed of unit cubes, $\Lambda_r^0$ -- the maximal one included in $B_r,$ and
$\Lambda_r^1$ -- the minimal one containing $B_r.$
Then,  by (i),
\begin{eqnarray*}
L(t) &=& \frac{ p(t,0,0) }{|\Lambda_r^0|}\int_{\Lambda_r^0}  \mathbb E^{\mathbb Q} \mathbf E_{x,x}^t\left[{\rm e}^{-\int_0^t V^\omega(Z_s){\rm d}s}   \right]\,{\rm d}x \\
&\leq & \frac{|B_r|}{|\Lambda_r^0|} \frac{ p(t,0,0) }{|B_r|}\int_{B_r}  \mathbb E^{\mathbb Q} \mathbf E_{x,x}^t\left[{\rm e}^{-\int_0^t V^\omega(Z_s){\rm d}s}   \right]\,{\rm d}x \\
&\leq& \frac{r^d}{(r-2\sqrt d)^d}  \frac{p(t,0,0) }{|B_r|}\int_{B_r}  \mathbb E^{\mathbb Q} \mathbf E_{x,x}^t\left[{\rm e}^{-\int_0^t V^\omega(Z_s){\rm d}s}   \right]\,{\rm d}x\\
&=& (1+o(1))\frac{ p(t,0,0) }{|B_r|}\int_{B_r} \mathbb E^{\mathbb Q} \mathbf E_{x,x}^t\left[{\rm e}^{-\int_0^t V^\omega(Z_s){\rm d}s}   \right]\,{\rm d}x
\end{eqnarray*}
and, identically,
\begin{eqnarray*}
L(t) &=& \frac{ p(t,0,0) }{|\Lambda_r^1|}\int_{\Lambda_r^1}  \mathbb E^{\mathbb Q} \mathbf E_{x,x}^t\left[{\rm e}^{-\int_0^t V^\omega(Z_s){\rm d}s}   \right]\,{\rm d}x \\
&\geq & \frac{|B_r|}{|\Lambda_r^1|} \frac{ p(t,0,0) }{|B_r|}\int_{B_r} \mathbb E^{\mathbb Q} \mathbf E_{x,x}^t\left[{\rm e}^{-\int_0^t V^\omega(Z_s){\rm d}s}   \right]\,{\rm d}x \\
&\geq& \frac{r^d}{(r+\sqrt d)^d}  \frac{ p(t,0,0) }{|B_r|}\int_{B_r}  \mathbb E^{\mathbb Q} \mathbf E_{x,x}^t\left[{\rm e}^{-\int_0^t V^\omega(Z_s){\rm d}s}   \right]\,{\rm d}x\\
&\geq& (1-o(1))\frac{ p(t,0,0) }{|B_r|}\int_{B_r}  \mathbb E^{\mathbb Q} \mathbf E_{x,x}^t\left[{\rm e}^{-\int_0^t V^\omega(Z_s){\rm d}s}   \right]\,{\rm d}x.
\end{eqnarray*}
\end{proof}

\section{The asymptotics} \label{sec:asymp}
We first present the results, postponing their proofs to the next section.
We will separately prove the upper- and the lower- bounds for the asymptotics of the Laplace transform of the integrated density of states at infinity, which we then  transform to statements
concerning the IDS itself. We will see that the behaviour of the integrated density of states depends decisively on the properties of the distribution of the random variables $q_{\mathbf i}$ at zero.
When $ F_q(0)=\mathbb Q[q_{\mathbf i}=0]>0$ (i.e. the distribution of the $q$'s has an atom at zero), then we obtain Lifschitz tail with rate
identical as that in the continuous Poisson-Anderson  model from \cite{bib:Oku}. On
the other hand, when this distribution has no atom at zero, then such a rate would be too small -- we obtain an infinite limit. In both cases we actually establish the existence of the limit $\lim_{\lambda \searrow 0}
\lambda^{d/\alpha}\ln \ell([0,\lambda]).$

\subsection{The upper bound for the Laplace transform}\label{sec:upper}
We start with the upper bound, which reads as follows.
\begin{theorem}\label{th:upper-stable}
For any $\kappa> 0$
 we have
\begin{equation}\label{eq:upper-statement}
	\limsup_{t\to\infty} \frac{\ln L(t)}{t^{\frac{d}{d+\alpha}}} \leq -C_{d,\alpha} \left(\ln\frac{1}{F_q(\kappa)}\right)^{\frac{\alpha}{d+\alpha}},
\end{equation}
where the constant $C_{d,\alpha}$ is given by
\begin{equation}\label{eq:opt-constant}
C_{d,\alpha} = \omega_d^{\frac{\alpha}{d+\alpha}}
\left(\frac{d+\alpha}{\alpha}\right)\left(\frac{\alpha \lambda_d^{(\alpha)}}{d}\right)^{\frac{d}{d+\alpha}}.
\end{equation}
In this formula, $\omega_d$ is the volume of the unit ball in $\R^d,$ and $\lambda_d^{(\alpha)}$  is the principal eigenvalue as in Theorem \ref{coro-no-atom-at-0_IDS}.
\end{theorem}

 From this statement, by passing to the limit  $\kappa \searrow 0,$  we immediately see that when the distribution of $q$ has no atom at zero, but its support is not included in any half-line
$[a_0,\infty)$ with $a_0>0,$ then the upper limit in \eqref{eq:upper-statement} is infinite. We can conclude that in that case
the 'Poissonian' rate $t^{\frac{d}{d+\alpha}}$ is too slow. In a companion paper \cite{bib:KK-KPP-alloy2} we use a different method to
identify the correct rate in that case, also for models driven by more general L\'{e}vy processes and with more general potential profiles, possibly unbounded and of noncompact support.

\begin{corollary}\label{coro-no-atom-at-0}
Suppose that the distribution of $q$ has no atom at zero, i.e. $F_q(0)=0.$ Then
	\begin{equation}\label{eq:no-limit}
\lim_{t\to\infty}\frac{\ln L(t)}{t^{\frac{d}{d+\alpha}}}=	 -\infty.
\end{equation}
	\end{corollary}

On the other hand, if the distribution of  $q$ has an atom at 0,  then passing to the limit  $\kappa \searrow 0$  leads to the following statement.

\begin{corollary}\label{coro-atom-at-0}
	Suppose that  the distribution of $q$ has an atom at 0, i.e. $ F_q(0)>0.$
	Then
	\[\lim_{t\to\infty}\frac{\ln L(t)}{t^{\frac{d}{d+\alpha}}}\leq -C_{d,\alpha} \left(\ln \frac{1}{ F_q(0)}\right)^{\frac{\alpha}{d+\alpha}}.
	\]
	\end{corollary}

\subsection{The lower bound for the Laplace transform}\label{sec:lower}
When the distribution of the $q_{\mathbf i}$'s has an atom at zero,
a matching lower bound is needed to obtain the existence of the limit
$\lim_{t\to\infty}\frac{\ln L(t)}{t^{\frac{d}{d+\alpha}}}.$ It is provided in the following theorem.
\begin{theorem}\label{th:dolne} Suppose that the distribution of $q$ satisfies
 \begin{equation}\label{assump-F}
 F_q(0)>0. \end{equation}
  Then
\begin{equation}\label{eq:preliminary_lower}
\liminf_{t\to\infty} \frac{\ln  L(t)}{t^{\frac{d}{d+\alpha}} }
\geq -C_{d,\alpha} \left({\ln \frac{1}{F_q(0)}} \right)^{\frac{\alpha}{d+\alpha}}.
\end{equation} Consequently, in this case
\begin{equation}\label{eq:limit}
\lim_{t\to\infty} \frac{\ln  L(t)}{t^{\frac{d}{d+\alpha}} }
 = -C_{d,\alpha} \left({\ln \frac{1}{F_q(0)}} \right)^{\frac{\alpha}{d+\alpha}}.
 \end{equation}
\end{theorem}

\subsection{The Lifschitz tail}
We conclude with the formal proof of the results on the asymptotic behavior of $\ell[0,\lambda]$ as $\lambda \searrow 0$ that are stated in Theorem \ref{coro-no-atom-at-0_IDS} in Introduction.

\begin{proof}[Proof of Theorem \ref{coro-no-atom-at-0_IDS}]
The result follows directly from statements \eqref{eq:no-limit}, \eqref{eq:limit} via the Fukushima Tauberian theorem of the exponential type \cite[Theorem 2.1]{bib:F}.
\end{proof}

\section{Proof of Theorem \ref{th:upper-stable}}\label{sec:upper-proof}

\subsection{ Periodization of the potential}

In this proof we will work with  processes on tori $\mathcal T_M=  \R^d/(M\Z^d)\approx [0,M)^d,$ where
$M=1,2,...$  By $\pi_M:\R^d\to \mathcal T_M$ we will denote the
canonical projection. To begin with,  we periodize the lattice random variables $\{q_{\mathbf i}\}_{\mathbf i\in \Z^d}$ with respect to $\pi_M$, and then, based on that, we construct a new random potential which we call the {\em Sznitman-type periodization} of the initial potential $V^{\omega}$. More precisely, for $M\geq 1$ we define
\begin{eqnarray}\label{eq:szn-per-of-V}
V_M^{\omega}(x)&:= &\sum_{{\bf i}\in[0, M)^d} \left(q_{\bf i}(\omega) \sum_{{\bf i'}\in \pi_{M}^{-1}({\bf i})} W(x-{\bf i'})\right) \nonumber\\
&=& \sum_{ {\bf i} \in \Z^d} q_{\pi_M({\bf i})} W(x-{\bf i}),\quad x\in\R^d.
\end{eqnarray}
The potential $V_M^{\omega}$ is periodic  in the usual sense:  for any $x\in\R^d,$ we have $V_M^{\omega}(x)=V_M^{\omega}(\pi_M(x)).$
For simplicity, we will use the same letter for the restriction of this potential to $\mathcal T_M.$

The following lemma  will be used for linking the periodized and the un-periodized potential.

\begin{lemma}\label{lem:periodic}
Let $(a_{\bf i})_{{\bf i\in\Z^d}}$ be  real numbers and let $(q_{\bf i})_{{\bf i\in\Z^d}}$
 be nonnegative, i.i.d. random variables  over the probability space $(\Omega, \mathcal A, \mathbb Q)$.  Then, for any $M=1,2,...$
 \begin{equation}\label{eq:two_potentials}
 \mathbb E^{\mathbb Q} \left[ {\rm e}^{-\sum_{{\bf i}\in\Z^d} a_{\bf i}q_{\bf i}}\right]
 \leq
 \mathbb E^{\mathbb Q} \left[ {\rm e}^{-\sum_{{\bf i}\in[0,M)^d} q_{\bf i}
 	 \sum_{{\bf i'}\in \pi_M^{-1}({\bf i})} a_{{\bf i'}}}
  \right].
 \end{equation}
 \end{lemma}

\begin{proof}
	Clearly, since the all the $a_{\bf i}$'s and the $q_{\bf i}$'s are nonnegative, it is enough the prove the
	statement for finite sums:
	 \[\mathbb E^{\mathbb Q} \left[ {\rm e}^{-\sum_{{\bf i}\in[-kM,kM)^d} a_{\bf i}q_{\bf i}}\right]
	\leq
	\mathbb E^{\mathbb Q} \left[ {\rm e}^{-\sum_{{\bf i}\in[0,M)^d}
	 q_{\bf i}	\sum_{{\bf i'}\in \pi_M^{-1}({\bf i})\cap [kM,kM)^d} a_{{\bf i'}}}
	\right]
	\]
	and then to pass to the limit $k\to\infty$ using dominated convergence.
	We have:
	\[\sum_{{\bf i}\in [-kM,kM)^d} a_{\bf i} q_{\bf i} = \sum_{{\bf i}\in[0,M)^d} \sum_{{\bf i'}\in \pi_M^{-1}({\bf i})\cap [-kM,kM)^d} a_{\bf i'} q_{\bf i'}.\]
	From this and the independence of the $q_{\bf i}$'s,
	\[\mathbb E^{\mathbb Q} \left[{\rm e}^{-\sum_{{\bf i}\in [-kM,kM)^d} a_{\bf i} q_{\bf i} }\right] = \prod_{{\bf i}\in [0,M)^d} \mathbb E^{\mathbb Q} \left[{\rm e}^{-\sum_{{\bf i'}\in \pi_M^{-1}({\bf i})\cap [-kM,kM)^d} a_{\bf i'} q_{\bf i'}}\right].
	\]
	Fix now a lattice point ${\bf i}\in[0,M)^d$ and label the lattice points in $ \pi_M^{-1}({\bf i})\cap [-kM,kM)^d$ as ${\bf i}_1,...,{\bf i}_{n_k}.$  Without loss of generality all the coefficients $a_{\bf i}$ are nonzero -- otherwise we just remove the corresponding lattice points.  We now apply the generalized H\"{o}lder inequality with $p_l= \frac{a_{{\bf i}_1}+...+a_{{\bf i}_{n_k}}}{a_{{\bf i}_l}},$  $l=1,...,n_k,$ to obtain
	\[\mathbb E^{\mathbb Q} \left[{\rm e}^{-(a_{{\bf i}_1}q_{{\bf i}_1}+...+ a_{{\bf i}_{n_k}} q_{{\bf i}_{n_k}})}\right]\leq
	\prod_{l=1}^{n_k}\left(\mathbb E ^{\mathbb Q} \left[{\rm e}^{-q_{{\bf i}_l}(a_{{\bf i}_1}+...+a_{{\bf i}_{n_k}})}\right]\right)^{\frac{a_{{\bf i}_l}}{a_{{\bf i}_1}+...+a_{{\bf i}_{n_k}}}}=
	\mathbb E ^{\mathbb Q} \left[{\rm e}^{-q_{{\bf i}}(a_{{\bf i}_1}+...+a_{{\bf i}_{n_k}})}\right],
	\]
	which proves the statement.
\end{proof}	

\begin{remark}{\rm At a first glance it might be confusing that on the right-hand side of \eqref{eq:two_potentials} we deal with infinite sums
$\sum_{\mathbf i'\in \pi_M^{-1}(\mathbf i)} a_{\mathbf i'} $ which, if infinite, may cause the entire right-hand side expression to be zero.
However, notice that in that case we would have $\sum_{\mathbf i'\in\pi_M^{-1}(\mathbf i)}q_{\mathbf i'}a_{\mathbf i'}=\infty$  a.s. as well.
 To see this, just apply the Kolmogorov's three-series theorem: if $X_n$ are i.i.d., nonnegative  and nondegenerate  random variables, and $a_n$ -- nonnegative numbers such that $\sum_n{a_n}=\infty$ then, denoting $a=\sup_n a_n,$
$$
\sum_{n}\mathbb E\left(a_nX_n\mathbf 1_{\{a_nX_n\leq 1\}}\right)
\geq \sum_{n}\mathbb E\left(a_nX_n\mathbf 1_{\{a X_n\leq 1\}}\right)=\mathbb E\left(X_1\mathbf 1_{\{a X_1\leq 1\}}\right) \sum_n a_n=\infty,
$$
which implies that $\sum_n a_nX_n=\infty$ a.s. }
\end{remark}

\subsection{Stable processes on tori and their Schr\"{o}dinger perturbations}

Our next tools will be  stable processes on tori $\mathcal T_M,$
$M\geq 1,$
defined  pathwise as
\[Z_t^M:= \pi_M(Z_t), \quad t\geq 0.\]
They will be again symmetric Markov  (and, in fact, Feller and strong Feller)  processes, with transition probability densities given by
\[p_M(t,x,y) = \sum_{y'\in\pi_M^{-1}(y)} p(t,x,y'),\quad t>0, \ x,y\in\mathcal T_M.\]
 Using the regularity properties and the estimates of the transition densities of the free process on $\R^d,$ we can deduce that for any $M\in\Z^d$, the function $p_M(\cdot,\cdot,\cdot)$ is continuous on $(0,\infty) \times \mathcal T_M \times \mathcal T_M$ and $p_M(t,\cdot,\cdot)$ is bounded on $\mathcal T_M \times \mathcal T_M$, for every fixed $t>0$. In particular, one can easily see that
\begin{align} \label{eq:diag-on-est}
 \sup_{x\in\mathcal T_M} p_M(1,x,x) = \sup_{x \in [0,M)^d}\sum_{\mathbf i\in M\Z^d} p(1,x, x+\mathbf i) = \sum_{\mathbf i\in M\Z^d} p(1,0, \mathbf i) \leq \sum_{\mathbf i \in \Z^d} p(1,0, \mathbf i) =: C_1
 \end{align}
(note that the constant $C_1$ does not depend on $M$).  For $\alpha \in (0,2)$ (the pure jump case), the L\'{e}vy kernel of the $\alpha-$stable process on $\mathcal T_M$ is given by
\begin{equation}\label{eq:def-nu-m}
\nu_M(x,y)=\sum_{y'\in\pi_M^{-1}(y)} \nu(x,y') = \sum_{y'\in\pi_M^{-1}(y)} \frac{\mathcal A_{d,-\alpha}}{|x-y'|^{d+\alpha}},\quad x,y\in\mathcal T_M, \ \ x \neq y.
\end{equation}

For $x\in \mathcal T_M, $ by  $\mathbf P^M_x$ we denote the family of measures on $D([0,\infty), \mathcal T_M)$ (the space of c\`adl\`ag, $\mathcal T_M-$valued functions with Skorohod topology) induced by the kernels $p_M(t,x,y),$  i.e. $\mathbf P^M_x(Z_t^M \in {\rm d}y) = p_M(t,x,y) {\rm d}y$ (as usual, $\ex^M_x$ is the corresponding expectation). As before, for $t>0$ and $x,y\in\mathcal T_M,$
$\mathbf P_{x,y}^{M,t}$ denotes the bridge measure of the process $Z_t^M,$ conditioned to have $ Z_t^M=y,$ $\mathbf P_x^M-$almost surely, defined by a relation similar to \eqref{eq:bridge}.

The bridge measures for the
process in $\R^d$ and on the torus $\mathcal T_M$ are related through the
following identity.

\begin{lemma}
	\label{lm:rotation}
	 For every $t>0$, $x,y \in \R^d$, $M=1,2,...$ and any set $A \in \mathcal B(D[0,t],\mathcal T_M)$ we have
	\begin{equation}\label{eq:rot1}
	p_M(t,\pi_M(x),\pi_M(y))\mathbf P^{M,t}_{\pi_M(x),\pi_M(y)}[A]=
	\sum_{y'\in\pi_M^{-1}(\pi_M(y))} p(t,x,y')\mathbf
	P^t_{x,y'}[\pi_M^{-1}(A)]\end{equation}	
\end{lemma}
This statement  is readily seen for cylindrical sets and then
extended to the desired range of $A$'s by the Monotone Class Theorem.
Its fractal counterpart was discussed in \cite[Lemma 2.6]{bib:KK-KPP1}.

For $f \in L^2(\mathcal T_M, {\rm d}x)$ and $t>0$ we let
\begin{eqnarray*}
\mathcal E^{M}_t(f,f)&=& \frac{1}{t}\left\langle f-T_t^{M}f,f\right\rangle,
\end{eqnarray*}
where $T_t^{M} f(x) = \ex^M_x f(Z^M_t)=\int_{\mathcal T_M} p_M(t,x,y) f(y) {\rm d}y$ denotes the action of the transition semigroup operator of the process $(Z_t^M)_{t \geq 0}$. For every fixed $f \in L^2(\mathcal T_M, {\rm d}x)$, $t \mapsto \mathcal E^{M}_t(f,f)$ is a nonnegative and nondecreasing function. The quadratic form corresponding to the process
$(Z_t^M)_{t \geq 0}$ is then defined by
$$
\mathcal E^{M}(f,f) = \lim_{t \searrow 0} \mathcal E^{M}_t(f,f), \quad f \in L^2(\mathcal T_M, {\rm d}x).
$$
The domain $D(\mathcal E^{M})$ of this form consists of those functions $f \in L^2(\mathcal T_M, {\rm d}x)$ for which this limit is finite.
We have the following close formulas:
\[\mathcal E^{M}(f,f)
= \frac{1}{2}\int_{\mathcal T_M}\int_{\mathcal T_M} (f(x)-f(y))^2\nu_M(x,y){\rm d}x{\rm d}y\]
for $\alpha\in (0,2)$ (the pure jump case), and
\[\mathcal E^{M}(f,f)
= \frac{1}{2}\int_{\mathcal T_M} |\nabla f(x)|^2 {\rm d}y,\]
for $\alpha=2$ (the Brownian motion case).

Throughout this section, we will also consider the Feynman-Kac semigroups of the projected $\alpha-$stable process on tori $\mathcal T_M$, $M \in \Z_+$. More precisely, for a periodized potential $V_M^{\omega}(x)$ (given by \eqref{eq:szn-per-of-V}) restricted to $\mathcal T_M$, we define the operators
$$
T_t^{M,V^{\omega}_M} f(x) = \ex^M_x\left[{\rm e}^{-\int_0^t V^{\omega}_M(Z^M_s) {\rm d}s} f(Z^M_t)\right], \quad f \in L^2(\mathcal T_M, {\rm d}x), \quad t>0.
$$
These operators are bounded and symmetric, and they form a strongly continuous semigroup on $L^2(\mathcal T_M, {\rm d}x)$. For every $t>0$, $T_t^{M,V^{\omega}_M}$ is an integral operator with the kernel
$$
p_{M}^{V^{\omega}_M}(t,x,y) = p_M(t,x,y) \ex_{x,y}^{M,t} \left[{\rm e}^{-\int_0^t V^{\omega}_M(Z^M_s) {\rm d}s} \right], \quad t> 0, \ x,y \in \mathcal T_M.
$$
As $V^\omega_M$ is bounded, the corresponding quadratic form is given by
$$
\mathcal E^{M, V_M^\omega}(f,f)
= \mathcal E^{M}(f,f) + \int_{\mathcal T_M} V_M^\omega(x) f^2(x){\rm d}x,\quad f\in\mathcal D(\mathcal E^{M, V_M^\omega}) = \mathcal D(\mathcal E^{M}).
$$
Since $\sup_{x,y \in \mathcal T_M} p_{M}^{V^{\omega}_M}(t,x,y) \leq  \sup_{x,y \in \mathcal T_M} p_M(t,x,y) < \infty$ and the measure of $\mathcal T_M$ is finite,
each $T_t^{M,V^{\omega}_M}$, $t>0$, is a Hilbert-Schmidt operator. In consequence, there exists a complete orthonormal set $\left\{\varphi_k^{M,V^{\omega}_M}\right\}_{k=1}^{\infty}$ in $L^2(\mathcal T_M,{\rm d}x)$, consisting of eigenfunctions of the operators $T_t^{M,V^{\omega}_M}$. More precisely, we have
$$
T_t^{M,V^{\omega}_M} \varphi_k^{M,V^{\omega}_M} = e^{-\lambda_k^{M,V^{\omega}_M} t} \varphi_k^{M,V^{\omega}_M}, \quad t> 0, \ k=1,2,...
$$
The corresponding eigenvalues satisfy $0\leq \lambda_1^{M,V^{\omega}_M} < \lambda_2^{M,V^{\omega}_M}\leq \lambda_3^{M,V^{\omega}_M} \leq \ldots \to \infty$; each $\lambda_k^{M,V^{\omega}_M}$ is of finite multiplicity.  The ground state eigenvalue $\lambda_1^{M,V^{\omega}_M}$ is simple and it can be expressed via the variational formula
\begin{equation}\label{eq:def-lambda-MV}
\lambda_1^{M,V_M^\omega}=\inf\{ \mathcal E^{M, V_M^\omega}(f,f): f\in L^2(\mathcal T_M, {\rm d}x), \|f\|_{2}=1\}.
\end{equation}

In Section \ref{sec:cont} we will also need the process $(Z_t^M)_{t \geq 0}$ killed on leaving an open and nonempty set $U \subset \mathcal T_M$ (this final application requires $V_M^\omega\equiv 0$ only). The transition semigroup of such a process consists of symmetric operators
$$
T_t^{M,U} f(x) = \ex^M_x\left[f(Z^M_t); t< \tau_U\right] = \int_{U} p_M(t,x,y) \ex_{x,y}^{M,t} \left[{\rm e}^{-\int_0^t V^{\omega}_M(Z^M_s) {\rm d}s}; t< \tau_U \right] f(y) {\rm d}y, \quad  t>0.
$$
By the same standard arguments as above, they are Hilbert-Schmidt in $L^2(U, {\rm d}x)$. This ensures the pure point spectrum and nondegeneracy of the ground state. Again, the ground state eigenvalue can be identified via the variational formula
\begin{equation}\label{eq:def=lambdaMU}
\lambda_1^M(U) = \inf\{ \mathcal E^{M}(f,f): f\in L^2(\mathcal T_M, {\rm d}x),\, \mbox{supp}\,f\subset U,\, \|f\|_{2}=1\}.
\end{equation}
Recall that this infimum is attained for unique $\varphi \in L^2(U, {\rm d}x)$ called the ground state eigenfunction.

\subsection{ Proof of Theorem \ref{th:upper-stable} -- the preparatory steps.} Suppose $M\in\mathbb Z_+$ is given.
We  use the representation of $L(t)$ from  Corollary \ref{coro-el-t-properties}(i):
\begin{equation}\label{eq:aab}
L(t)=\frac{ p(t,0,0)}{M^d}\int_{[0,M)^d}\mathbb E^{\mathbb Q} \mathbf E_{x,x}^t \left[{\rm e}^{-\int_0^t V^\omega(Z_s){\rm d}s}\right]{\rm d}x.
\end{equation}
 For any given realization of the free process, $Z_s(w),$
we can apply Lemma \ref{lem:periodic} with coefficients $a_{\bf i}= \int_0^t W(Z_s(w)- {\bf i})\,{\rm d}s$  to  get that
\[\mathbb E^{\mathbb Q}  \left[{\rm e}^{-\int_0^t V^\omega(Z_s(w)){\rm d}s} \right]\leq \mathbb E^{\mathbb Q}  \left[{\rm e}^{-\int_0^t  V_M^{\omega}(Z_s(w)){\rm d}s}\right],\]
therefore we can swap $V_M^{\omega}$ for $V^\omega$ in \eqref{eq:aab}, getting (for any fixed value of $M$)
\[
L(t)\leq\frac{1}{M^d}\int_{[0,M)^d} p(t,x,x)\mathbb E^{\mathbb Q} \mathbf E_{x,x}^t \left[{\rm e}^{-\int_0^t  V_M^{\omega}(Z_s){\rm d}s}\right]{\rm d}x.
\]
Further, we want to replace the process  $(Z_t)_{t \geq 0}$ on $\R^d$ with the projected process on $\mathcal T_M$ --  the transition density $p(t,x,x)$ and the bridge measures $\mathbf E_{x,x}^t$ will replaced by their torus-counterparts $p_M(t,x,x)$ and $\mathbf E^{M,t}_,$ pertaining to the projected process  $(Z_t^M)_{t \geq 0}$  on the torus.
This procedure can only increase the integral,
which follows from Lemma \ref{lm:rotation}. Indeed, since the potential $V_M^{\omega}$ is periodic, for any $x\in [0,M)^d$ the expression
$p(t,x,x)\mathbf E_{x,x}^t\left[{\rm e}^{-\int_0^tV_M^{\omega}(Z_s){\rm d}s}\right]$
is just the first term in the expansion of
$p_M(t,x,x)\mathbf E_{x,x}^{M,t}\left[{\rm e}^{-\int_0^tV_M^{\omega}(Z_s^M){\rm d}s}\right]$ arising from \eqref{eq:rot1}.
We arrive at
\[
L(t)\leq\frac{1}{M^d}\int_{\mathcal T_M} p_M(t,x,x)\mathbb E^{\mathbb Q} \mathbf E_{x,x}^{M,t} \left[{\rm e}^{-\int_0^t V_M^{\omega}(Z_s^M){\rm d}s}\right]{\rm d}x.
\]

Let now $\kappa>0$  be fixed.
The functional  $e_{V_M^{\omega}}(w,t)= {\rm e}^{-\int_0^t V_M^{\omega}(Z_s^M(w))\,{\rm d}s}$ will not be diminished if we
replace the random variables $q_{\bf i}$ with
new random variables $q_{\bf i}^{(\kappa)}\leq q_{\bf i},$ defined as
\[
q_{\bf i}^{(\kappa )}= \left\{\begin{array}{ll}
0 &\mbox{if } q_{\bf i}(\omega) \leq \kappa,\\
\kappa & \mbox{if }q_{\bf i}(\omega)> \kappa.
\end{array}
\right.
\]

\noindent
We denote the resulting potential by $V_M^{\kappa}$ (we drop the dependence on $\omega$ from now on).		
After all these operations we have:
 \begin{eqnarray*}
  L(t) &\leq& \frac{1}{M^d}\,\mathbb E^{\mathbb Q}\int_{\mathcal T_M} p_M(t,x,x)\mathbf E^{M,t}_{x,x}
	\left[ {\rm e}^{-\int_0^t V^{\kappa}_M(Z_s^M)\,{\rm d}s}\right]\,{\rm d}x\\
	&=& \frac{1}{M^d}\, \mathbb E^{\mathbb Q}\mbox{Tr}\,  T_t^{M,V^{\kappa}_M}  = \frac{1}{M^d}\mathbb E^{\mathbb Q}\sum_{n=1}^\infty {\rm e}^{-t \lambda_n^{M,V^{\kappa}_M}}
\end{eqnarray*}
we have $0 \leq \lambda_1^{M,V^{\kappa}_M} <  \lambda_2^{M,V^{\kappa}_M} \leq  \lambda_3^{M,V^{\kappa}_M} \leq ...$, it follows that
for $M=1,2,...$ and $t>1$
\begin{eqnarray*}
	L(t) & \leq & \mathbb{E}^{\qpr} \left[{\rm e}^{-(t-1)  \lambda_1^{M,V^{\kappa}_M}} \cdot \frac{1}{M^d} \sum_{n=1}^\infty {\rm e}^{-\lambda_n^{M,V_M^{\kappa}}}\right]
	\\
	&=& \mathbb{E}^{\qpr}\left[ {\rm e}^{-(t-1)  \lambda_1^{M,V^{\kappa}_M}} \cdot \frac{1}{M^d}\,
	\mbox{Tr} \, T_1^{M,V^{\kappa}_M}\right]
	\\ & \leq & \mathbb{E}^{\qpr} \left[{\rm e}^{-(t-1)  \lambda_1^{M,V^{\kappa}_{M}}}\right]  \cdot \frac{1}{M^d} \int_{\mathcal T_M} p_M(1,x,x) {\rm d}x,
\end{eqnarray*}
and using \eqref{eq:diag-on-est}, we can write
\begin{equation}\label{eq:el-em1}
L(t)\leq \frac{C_1}{M^d}\, \mathbb{E}^{\qpr} \left[{\rm e}^{-(t-1)  \lambda_1^{M,V^{\kappa}_M}}\right].
\end{equation}

To proceed, we need to obtain a nice estimate on $\lambda_1^{M,V^{\kappa}_M}.$  To this end, we will
employ Sznitman's enlargement of obstacles technique (the `coarse-graining' procedure). As the last preparatory step, we perform the scaling.

	Let $ m\in\mathbb Z_+$ be given. It will be considered fixed in this stage of the  proof.  Consider $M$'s being multiples of $m,$  i.e.  $M=Km, $ $K=1,2,...$
Afterwards, the number $K$ will depend on $t$  and will eventually tend to infinity (so far, the estimate \eqref{eq:el-em1} is valid for any $t>1$ and any $M\in\mathbb Z_+).$ Recall also that the parameter $\kappa>0$ is fixed throughout the entire proof.

	We intend to reduce our situation  to
	the problem on the torus $\mathcal T_m\approx [0,m)^d,$ equipped with the normalized Lebesgue measure, which we denote by $|\cdot|_m.$
	 Let
	\[
	 \widetilde V^{\kappa}_{K,m}(x):= K^\alpha V^{\kappa}_{Km}(Kx),\quad x\in \mathcal T_m\]
	(recall that $\alpha \in (0,2]$ is the stability index of the process).
This potential can be explicitly written as:
\begin{equation}\label{eq:pot-scaled}
	\widetilde V^{\kappa}_{K,m}(x)=K^\alpha\sum_{\mathbf i\in \Z^d: \frac{\mathbf i}{K}\in\mathcal T_m} q_{\pi_M(\mathbf i)}^{(\kappa)}(\omega)  W(Kx-{\mathbf i}),\quad x\in\mathcal T_m.
\end{equation}

	From the scaling of  the $\alpha-$stable process (see \cite[Lemma 4.3]{bib:KK-KPP2} for a general argument presented in a fractal setting) we deduce, for $M=Km,$
	\begin{equation}\label{eq:lambda-scal}
	\lambda_1^{M, V^{\kappa}_M} = 	 \lambda_1^{Km,V^{\kappa}_{Km}} = \frac{1}{K^\alpha}\,\lambda_1^{m,\widetilde{V}^{\kappa}_{K,m}}.
	\end{equation}
Indeed, to prove \eqref{eq:lambda-scal}, first let $f\in L^2(\mathcal T_M),$ $\|f\|_{L^2(\mathcal T_M)}$ be the ground state for $H^{M,V^\kappa_M};$ from the variational formula \eqref{eq:def-lambda-MV}
we have $\mathcal E^{M,V_M^\kappa}(f,f)= \lambda_1^{M,V_M^\kappa}.$
Let now $g(x)= K^{d/2}f(Kx),$ $x\in \mathcal T_m.$
We immediately verify that $\int_{\mathcal T_m} g^2(x){\rm d}x=\int_{\mathcal T_M} f^2(x){\rm d}x=1,$
and
\begin{equation}\label{eq:f1}
\int_{\mathcal T_m} g^2(x) \widetilde V^\kappa_{K,m}(x){\rm d}x =K^\alpha \int_{\mathcal T_M} f^2(x) V^\kappa_M(x){\rm d}x.
\end{equation}
Moreover,
\begin{eqnarray}\label{eq:f2}
\int_{\mathcal T_m}\int_{\mathcal T_m}(g(x)-g(y))^2 \nu_m(x,y){\rm d}x{\rm d}y & = & \int_{\mathcal T_m}\int_{\mathcal T_m}K^d(f(Kx)-f(Ky))^2 \nu_m(\frac{Kx}{K},\frac{Ky}{K}){\rm d} x{\rm d}y\nonumber\\
&=& K^\alpha\int_{\mathcal T_M}\int_{\mathcal T_M}(f(x)-f(y))^2\nu_M(x,y){\rm d}x{\rm d}y
\end{eqnarray}
(in the last line we have used the explicit formula \eqref{eq:def-nu-m} for the L\'{e}vy density of the $\alpha-$stable processes on $\mathcal T_m$ and $\mathcal T_M$).
Adding identities \eqref{eq:f1} and \eqref{eq:f2} we get
\[\mathcal E^{m, \widetilde V_{K,m}^\kappa}(g,g)= K^\alpha \mathcal E^{M, V^\kappa_M}(f,f)=K^\alpha\lambda_1^{M,V_M^\kappa},\]
and the variational formula yields $\lambda_1^{m,\widetilde V_{K,m}^\kappa}\leq K^\alpha \lambda_1^{M, V_M^\kappa}.$

To get the opposite inequality we proceed identically, starting this time with  $g$ being the ground state for $H^{m,\widetilde V^{\kappa}_{K,m}}.$

	\subsection{ Enlargement of obstacles technique in present case}
To make the article self-contained, we sketch here a version of Sznitman's theorem, proven in present setting in the Appendix of \cite{bib:KK-KPP2}.
The setup for this theorem consists of:
\begin{itemize}
\item[$*$] a compact linear metric space $({\mathcal T}, d)$ equipped with a doubling Radon measure $m,$ satisfying $m({\mathcal T})=1$.
    More precisely, we assume that there exist $r_0>0$ and $C_d\geq 1$ such that for any $x\in\mathcal T$ and $0<r<r_0$
    \begin{eqnarray}\label{doubling}  m(B(x,r))&\leq& C_d m(B(x, \frac{r}{3})),\end{eqnarray}
\item[$*$] a right-continuous, strong Markov process $X=\left(X_t, \pr_x\right)_{t \geq 0, \, x \in \mathcal T}$ on $\mathcal T$ with symmetric and strictly positive
transition density
 $p(t,x,y)$ with respect to $m$ such that $\forall\,\, t,$ \linebreak $\int_{\mathcal T} p(t,x,x){\rm d}m(x)<\infty,$
\item[$*$] a potential profile $W:{\mathcal T}\times{\mathcal T}\to\mathbb{R}_+$ of finite range:
 a measurable function with support included in
$\{(x,y)\in\mathcal T\times\mathcal T: d(x,y)\leq a\epsilon\},$ where
$a>0, \epsilon>0$ are given, such that
\begin{equation} \label{eq:finiteW}
\mbox{for every} \quad t>0 \ \ \mbox{and} \ \  y \in {\mathcal T}, \quad \sup_{x \in {\mathcal T}} \ex_x \int_0^t W(X_s,y)\,{\rm  d}s < \infty.
\end{equation}

In applications, $a$ will be considered fixed and $\epsilon$ will tend to $0.$

\item[$*$] A finite collection of points $x_1,...,x_N\in \mathcal {T},$ called `obstacles' and the potential $V(x)$ defined as follows:
 \begin{equation}\label{vee} \mathcal{T}\ni x\mapsto
V(x)= \sum_{i=1}^N W(x,x_i). \end{equation}
\end{itemize}

We study the process $X$ perturbed by the potential $V$. Formally, we consider the Feynman-Kac semigroup $(T^V_t)_{t \geq 0}$ on $L^2({\mathcal T}, m)$ consisting of symmetric operators
$$
T^V_t f(x) = \ex_x \left[e^{-\int_0^t V(X_s)\,{\rm d}s } f(X_t) \right], \quad f \in L^2({\mathcal T}, m), \ \ t > 0.
$$
Denote by $\lambda_1^V$ the bottom of the spectrum of the  (positive definite) operator $-A^V$, where $A^V$ is the generator of this process.
of this semigroup.

\smallskip

\noindent We need to assume the following  conditions regarding
the process $X$ and the potential profile~$W.$

\begin{enumerate}
\item[{\bf (P1)}]  There exists  $c_0>0 $ such that $ \sup_{x,y\in\mathcal T} p(1,x,y)\leq c_0.$
\end{enumerate}

 Note that under {\bf(P1)} all the operators $T^V_t$, $t>0$, are compact in $L^2({\mathcal T}, m)$ and, in consequence, $\lambda_1^V$ becomes an isolated and simple eigenvalue.

The remaining assumptions are concerned with recurrence properties
of the process. We require that for any fixed $a,b,$ such that  $a\ll b$, $b\epsilon<r_0$ ($a\epsilon$ is the range of the potential profile)
and  $\delta>0$ there exist constants $\tau_0, c_1, c_2,$ $
c_3,\alpha,\kappa>0,$ $R>3$ and a nonincreasing function
$\phi:(0,r_0)\to (0,1]$  (the constants and $\phi$ do depend on $a,b$ but not on $\epsilon$) such that:
\begin{enumerate}
\item[{\bf (P2)}] for $x,y\in\mathcal{T}$ with $d(x,y)\leq b\epsilon$ one
has
\[\mathbf{P}_x[\tau_{B(y,10 (R-2) b\epsilon)}<\frac{\tau_0\epsilon^\alpha}{2}]<c_1;\]
\item[{\bf (P3)}]  when
$x,y\in\mathcal{T},$ and $d(x,y)\leq b\epsilon,$  then  $$
\mathbf{E}_x\big[{\rm e}^{-\int_0^{(\tau_0\epsilon^\alpha)/2} W( X_s,y)\,{\rm
d}s}\big]\leq 1-2c_1;$$
\item[{\bf (P4)}] for $x,y\in\mathcal T$ satisfying $d(x,y)\leq  r\epsilon\leq
r_0$ one has
\[\mathbf{P}_x\big[T_{B(y, b\epsilon)}
\leq\frac{\tau_0\epsilon^\alpha}{2}]\geq \phi(r);\]
\item[{\bf (P5)}] for
 $10 b\epsilon \leq \beta\leq\frac{r_0}{R},$
  any points $x,y\in\mathcal{T}$ with  $d(x,y)\leq\beta,$ and for
any compact subset $E\subset\mathcal{T}$ satisfying
$m(E\cap\overline{B}(y,\beta))\geq \delta/C_d \cdot
m(\overline{B}(y,\beta)) $ one has
\[\mathbf{P}_x[T_E<\tau_{B(y, R\beta)}]\geq c_2;\]
\item[{\bf (P6)}] for $r<r_0/3,$ $\rho>3r$ and  $x,y\in\mathcal{T}$ satisfying
$d(x,y)\leq r$ one has
\[\mathbf{P}_x\left[X_{\tau_{B(y,r)}}\notin B(y, \rho)\right]\leq c_3\left(\frac{r}{\rho}\right)^{\kappa}.\]
\end{enumerate}

Next, for given  $A>\delta>0$  define
\begin{equation}\label{ce-em-delta}
C(A,\delta)= {\rm e}^A\left(1+c_0(1+\frac{A}{\delta})\right),
\end{equation}
where $c_0$ is the constant from {\bf(P1)}. We require
that the number $R$ entering assumptions {\bf (P2), (P5)}
satisfies
\begin{equation}\label{erzero}
 \frac{c_3}{R^{\kappa}-1}\leq \frac{1}{8} \,C(A,\delta)^{-1}.
\end{equation} This can be done without loss of generality:
if {\bf (P2), (P5)} are satisfied with certain $R>0,$ then
they are satisfied for any $\tilde R>R.$

We  perform the following operation: for given  $b\gg a$ we want to replace the support
of the potential $V$ by a much larger set $\bigcup_{i=1}^N
\overline B(x_i,b\epsilon),$ and then we kill the initial process $X$ when it enters this bigger
set. We are interested in comparing  the principal eigenvalue   of this process and the principal eigenvalue $\lambda_1^V$ of the process $X$ perturbed by the potential $V$.
In general, we cannot enlarge every
obstacle  -- we keep only
those obstacles $x_i$ that are well-surrounded by other obstacles
(so-called {\em good obstacles}, see below).  Other obstacles will
be disregarded. Formally, we consider the sets
\begin{eqnarray}\label{thetabe}
\mathcal O_b =\bigcup_{x_i-{\rm good}} \overline{B}(x_i,b\epsilon), &&
\Theta_b=\mathcal T\setminus\mathcal{O}_b.
 \end{eqnarray}
 The process evolves now in the open set $\Theta_b$ and is killed when it enters $\mathcal O_b.$
Denote by $\lambda_1(b)$ the smallest eigenvalue of the operator $-A_b$, where $A_b$ is the generator

The distinction between `good' and `bad' points is made as
follows.

\begin{definition}\label{defi:good-bad}
 Suppose $b,\delta$ are given, and $R>0$ is the number from the assumptions above, satisfying \eqref{erzero}. Let $x_1,...,x_N$ be given obstacle points. Then $x_{i_0}$ is called a good obstacle point if for all balls
$C=B(x_{i_0}, 10b\epsilon R^l)$ one has
\begin{equation}\label{good}
m\left(\bigcup_{i=1}^N \overline{B}(x_i,b\epsilon)\cap C\right)\geq
\frac{\delta}{C_d}\, m(C),
\end{equation}
($C_d$ is the constant from (\ref{doubling})) for all $l=0,1,2,...,$
as long as $10b\epsilon R^l<r_0.$ Otherwise, $x_{i_0}$ is called a bad obstacle point.
\end{definition}
Formally speaking, this notion depends on $b,\delta,R,$ but for
the time being we do not incorporate these parameters into the
notation.

 Balls with centers at bad obstacle points sum up to a
set with small volume.

\begin{lemma}\emph{\cite[Lemma 1.3]{bib:Szn1}}\label{lembad}
\begin{equation}\label{badpoints}
m\big(\bigcup_{x_i-{\rm bad}}\overline{B}(x_i, b\epsilon)\big)\leq \delta.
\end{equation}
\end{lemma}

We will employ the following theorem, comparing  the principal eigenvalue of the potential problem with  the principal eigenvalue of the obstacle problem.

\begin{theorem} \emph{\cite[Theorem A.1]{bib:KK-KPP2}}\normal\label{th:compare}  Let the numbers
$A>\delta>0,$ $ b\gg a$ be given.
Assume that the process $X_t$ is a discontinuous c\`adl\`ag process satisfying {\bf (P1)} -- {\bf (P6)}, with $R$ satisfying (\ref{erzero}).

   Then there exists $\epsilon_0=
\epsilon_0(a,b,\delta,A,R,c_0,c_1,c_2,c_3,\alpha,\kappa)$  such that for any $\epsilon<\epsilon_0$ ($b \epsilon$ is
the  radius of obstacles in (\ref{thetabe})) one has
\begin{equation}\label{lambda}
\lambda_1(b)\wedge A\leq \lambda_1^V\wedge
A+\delta.
\end{equation}
\end{theorem}
\

\subsection{Continuation of the proof of Theorem \ref{th:upper-stable}} \label{sec:cont}
For a fixed  configuration $q_{\mathbf i}(\omega),$  we have the following.
  The state-space for the Sznitman's theorem is $\mathcal T=\mathcal T_m,$  the measure is the normalized Lebesgue measure on $\mathcal T_m,$ denoted $|\cdot|_m,$ the process
is the stable process  $(Z_t^m)_{t \geq 0},$  the potential profile is $W_K(x,y)=K^\alpha \kappa W(Kx-Ky),$ the obstacle points -- those
 points $\mathbf j=\frac{\mathbf i}{K}\in\mathcal T_m$ for which $q_{\mathbf i}^{(\kappa)}=\kappa.$  Denote the obstacle points (on the torus $\mathcal T_m)$  by $x_1,...,x_N.$ The range of this profile is equal to $\frac{a}{K},$ and we put $\epsilon=\frac{1}{K}$ (and we will use both $K,\epsilon$ below).

The potential, given by \eqref{eq:pot-scaled}, can be written as
\[\tilde{V}_{K,m}^{\kappa}(x)= \sum_{\mathbf j=\frac{\mathbf i}{K}  \mbox{\tiny  - an obstacle}} K^\alpha \kappa W(K(x-\mathbf j)).\]
denoted $V_\epsilon$ for short.
As in \eqref{eq:lambda-scal}, $\lambda_1^{m,\widetilde V^\kappa_{K,m}}=:\lambda_1^{V_{\epsilon}}$ is the principal eigenvalue of the stable
	semigroup on $\mathcal T_m$ with the potential $\widetilde V^\kappa_{K,m}=:V_\epsilon$
	-- this is the quantity we need to estimate.

		Fix two control levels  $A\gg 0$ and $\delta>0.$ Let $b=\beta\sqrt d \gg a ,$ $\beta\in\mathbb Z,$  be fixed  $(a$ is the range of the unscaled profile function $W$), and let  $R\gg 0$  be the number arising from {\bf (P2), (P5)},
satisfying \eqref{erzero}.
	
The good and bad obstacle points are defined as in the previous section.

	As before, we consider the sets
	\begin{eqnarray}\label{thetabe_2}
	\mathcal O_b =\bigcup_{x_i-{\rm good}} \overline{B}(x_i,b\epsilon), &&
	\Theta_b=\mathcal T_m\setminus\mathcal{O}_b
	\end{eqnarray}
and let the process evolve in $\Theta_b$ until it hits $\mathcal O_b.$ We denote by $\lambda_1(b)(=\lambda_1^m(\Theta_b))$ the smallest eigenvalue of the operator $-A_b$, where $A_b$ is the generator of this process.

Assumptions {\bf (P1)} -- {\bf (P6)} except for {\bf (P3)} were verified in \cite[Proposition 3]{bib:Kow}. To see {\bf (P3)} {we  argue as in \cite[Proposition 4.3]{bib:KK-KPP2}.} Consequently, we  now apply Theorem \ref{th:compare}, so that for given  $A>\delta>0,$  there
		exists $\epsilon_0=\epsilon_0(X, a, b, A,,R, \delta, \kappa)$ such that for any $\epsilon<\epsilon_0$  (or: there exists $K_0,$ depending on the same set of parameters, that for any $K>K_0$) one has
		\begin{equation}\label{eq:szn1}
\lambda_1(b)\wedge A \leq \lambda_1^{V_\epsilon}\wedge A+\delta.
\end{equation}
We recall that the principal eigenvalues in this statement pertain to the $\alpha-$stable process on the torus $\mathcal T_m.$

For $K>K_0$ described above, \eqref{eq:szn1} holds, and so the main estimate, for $M=Km,$ continues as  \eqref{eq:lambda-scal}
	\begin{eqnarray*}
	{\rm e}^{-(t-1)\lambda_1^{M,\widetilde{V}_{M}^{\kappa}}} & = & {\rm e}^{-\frac{(t-1)}{K^\alpha}\, \lambda_1^{m,\widetilde{V}_{K,m}^{\kappa}}}
		= {\rm e}^{-\frac{(t-1)}{K^\alpha}\, \lambda_1^{V_\epsilon}} \leq {\rm e}^{-\frac{(t-1)}{K^\alpha}\,(\lambda_1(b)\wedge A -\delta)}.
	\end{eqnarray*}
	
\

	To proceed, we discretize the set $\mathcal O_b.$ To this end, we chop the `sides' of the torus $\mathcal T_m$  into $\frac{m\sqrt d}{b\epsilon}= \frac{Km\sqrt d}{b}$  parts, which yields $ ( \frac{Km\sqrt d}{b})^d$ small boxes, with `sidelength' $b\epsilon/\sqrt d=b/(K\sqrt d)$ each, i.e. the diameter of those boxes is $b\epsilon.$ We do this in the manner that keeps the lattice points $\frac{\mathbf i}{K}$ inside the boxes. We can visualize this procedure as follows: we identify the torus with the box $[-\frac{1}{2K}, m-\frac{1}{2K})^d,$ and then chop the sides of this box into $\frac{Km\sqrt d}{b}=\frac{Km}{\beta}$ parts. Some of these boxes will now be  removed.
	
Let $U_{b,K}= U_{b,\epsilon}$ be the open subset of $\mathcal T_m$ obtained by removing closed boxes
that received some of the obstacle points. Similarly, the set $\widehat{U}_{b,K}=\widehat{U}_{b,\epsilon}\subset \mathcal T_m$ is   obtained by removing those closed boxes that received some good obstacle points.
	Finally,
let $\mathcal U_{b,\epsilon}$ be the set of all possible configurations of the sets $U_{b,\epsilon}, \widehat{U}_{b,\epsilon}$ (i.e. the set of all pairs $(U,\widehat U)$, $U, \widehat U \subset \mathcal T_m$, composed of boxes of size $b\epsilon$ that realize the random sets $U_{b,\epsilon}, \widehat{U}_{b,\epsilon}$).
Observe that if $(U,\widehat{U})\in \mathcal U_{b,\epsilon}, $ then $U\subset \widehat{U},$ and all the obstacles in $\widehat{U} \setminus U$  are bad. Clearly, if $\Delta $ a box with diameter $b\epsilon$ and $x\in \Delta,$ then $\Delta\subset B(x,b\epsilon).$ Therefore, if $U_{b,\epsilon}=U$ and $\widehat{U}_{b,\epsilon}=\widehat{U}$, then
 $$\Theta_b  \subset  U  \quad \mbox{ and } \quad \widehat{U} \setminus U \subset \bigcup_{x_i-\mbox{\tiny bad}} \overline{B(x_i,b\epsilon)}.$$
From
	 Lemma \ref{lembad} it follows that
	\begin{equation}
	|\widehat {U}|_m\leq |U|_m+  \delta,
	\end{equation}
	i.e.
	\begin{equation}\label{eq:measure-of-u}
|\widehat {U}|\leq |U|+ m^d \delta.
\end{equation}
For $(U,\widehat{U})\in\mathcal U_{b,\epsilon}$ we can now estimate $\mathbb Q[U_{b,\epsilon}=U, \widehat{U}_{b,\epsilon}=\widehat{U}].$
Since we have assumed that $b/\sqrt d=\beta\in\mathbb Z,$ the set $U$ consists of $|U|\cdot \epsilon^{-d}$ boxes of size $\epsilon= 1/K$ each, with precisely one lattice point in its center.

The condition $U=U_{b,\epsilon}$ means that  no lattice point inside   $U$ is an  obstacle , i.e. for all $\frac{\bf i}{K}\in U$ one has $q_{{\bf i}}(\omega)\leq \kappa.$ The probability of	 this event is
	\begin{eqnarray}\label{eq:est1}
\mathbb Q[U_{b,\epsilon}=U, \widehat{U}_{b,\epsilon}=\widehat{U}]&\leq&
	\mathbb Q\left[\mbox{no obstacles in $U$}\right]\nonumber\\
 &=& \mathbb Q[q\leq \kappa]^{K^d|U|}=
	(F_q(\kappa))^{K^d|U|}\leq (F_q(\kappa))^{K^d(|\widehat{U}|-m^d\delta)}\nonumber\\
	&=& {\rm e}^{-\ln(\frac{1}{F_q(\kappa)})K^d(|\widehat{U}|-m^d\delta)}.
	\end{eqnarray}

The cardinality of the family $\mathcal U_{b,\epsilon},$ denoted by $N_{b,\epsilon}$, satisfies
	\begin{equation}\label{eq:card-n}
	N_{b,\epsilon}\leq (2^{(\frac{Km\sqrt d}{b})^d})^2=2^{2(\frac{Km\sqrt d}{b})^d}.
	\end{equation}

\noindent Estimates \eqref{eq:measure-of-u}, \eqref{eq:est1}, \eqref{eq:card-n} account for the following chain of inequalities:
	\begin{eqnarray}\label{eq:esti-1}
	&&\mathbb E^{\mathbb Q}\left({\rm e}^{-\frac{(t-1)}{K^\alpha}\left[(\lambda_1(b)\wedge A)-\delta\right]}\right)
	\nonumber\\
	&\leq & \sum_{(U,\widehat{U})\in \mathcal U_{b,K}} \mathbb E^{\mathbb Q}\left[{\rm e}^{-\frac{(t-1)}{K^\alpha}[\lambda_1^m(\widehat{U})\wedge A-\delta]}
	\mathbf 1_{\{U_{b,\epsilon}=U, \widehat{U}_{b,\epsilon}=\widehat{U}\}}\right]\nonumber\\
	&=&
	\sum_{(U,\widehat{U})\in \mathcal U_{b,\epsilon}} \mathbb{\rm e}^{-\frac{(t-1)}{K^\alpha}[\lambda_1^m(\widehat{U})\wedge A-\delta]}
	\qpr\left[U_{b,\epsilon}=U, \widehat{U}_{b,\epsilon}=\widehat{U}\right]\nonumber\\
&\leq& \sum_{(U,\widehat{U})\in \mathcal U_{b,\epsilon}}  \mathbb{\rm e}^{-\frac{(t-1)}{K^\alpha}[\lambda_1^m(\widehat{U})\wedge A-\delta]-\ln\left(\frac{1}{F_q(\kappa)}\right)K^d(|\widehat{U}|-m^d\delta)}
\nonumber\\
&\leq &	2^{2(\frac{Km\sqrt d}{b})^d}\exp\left\{-\inf_{\widehat{U}\subset \mathcal T_m}\left[\frac{(t-1)}{K^\alpha}[\lambda_1^m(\widehat{U})\wedge A-\delta]+\ln(\frac{1}{F_q(\kappa)})K^d(|\widehat{U}|-m^d\delta)\right]\right\},
	\end{eqnarray}
	where the infimum runs over all possible configurations of $\widehat{U}.$

	Since for $a,b,c>0$ we have $a\wedge b+c \geq (a+c)\wedge b,$ we get that the quantity in \eqref{eq:esti-1}
	is not smaller that
	\[
	 2^{2(\frac{Km\sqrt d}{b})^d}\exp\left\{-\left(\inf_{\widehat{U}\subset \mathcal T_m}\left[\frac{(t-1)}{K^\alpha}\lambda_1^m(\widehat{U}) +\ln(\frac{1}{F_q(\kappa)})K^d|\widehat{U}|\right]\right)\wedge A + \delta\left(\frac{(t-1)}{K^\alpha} +m^d \ln (\frac{1}{F_q(\kappa)})\right)\right\}.
	\]	
We now work with the infimum in the formula above.  If we replace the condition
$\widehat U\subset \mathcal T_m$ with $\widehat U\in \mathcal U_m$ (the collection of all open subsets of $\mathcal T_m),$  it can only get diminished, i.e.
\[\inf_{\widehat{U}\in \mathcal T_m}\left[\frac{(t-1)}{K^\alpha}\lambda_1^m(\widehat{U})
+\ln(\frac{1}{F_q(\kappa)})K^d|\widehat{U}|\right]
\geq \inf_{\widehat{U}\subset \mathcal U_m}\left[\frac{(t-1)}{K^\alpha}\lambda_1^m(\widehat{U})
+\ln(\frac{1}{F_q(\kappa)})K^d|\widehat{U}|\right].
\]
Our bounds hold for any fixed $A,b,\delta,$ as long as $K$ is large enough. We now will make $K$ depend on $t.$
		If
	\begin{equation}\label{eq:choice-of-K}
	K_0\leq \left(\frac{(t-1)}{-\ln F_q(\kappa)}\right)^{\frac{1}{d+\alpha}}< (K_0+1),
	\end{equation}
	then
	we choose $K=K_0$ in the bounds above.
	
	 We introduce and fix  an additional control parameter $\zeta\in(0,1).$
	There exists $t_0=t_0(\zeta)$ such that for $t>t_0$ we have $\left(\frac{K_0}{K_0+1}\right)^d \geq 1-\zeta,$ so that	for $t\geq t_0$ we have
	\begin{eqnarray*}
&&
	\inf_{ \widehat{U}\subset \mathcal U_m}\left[\frac{(t-1)}{K_0^\alpha}\lambda_1^m(\widehat{U})
	+\ln(\frac{1}{F_q(\kappa)})K_0^d|\widehat{U}|\right]\\
& \geq & \inf_{ \widehat{U}\subset \mathcal U_m}\left[(t-1)^{\frac{d}{d+\alpha}}\left(\ln\frac{1}{F_q(\kappa)}\right)^{\frac{\alpha}{d+\alpha}}\left(
	\lambda_1^m(\widehat U) + \left(\frac{K_0}{K_0+1}\right)^d  |\widehat U| \right)\right]\\
	&\geq &   (t-1) ^{\frac{d}{d+\alpha}}\left(\ln\frac{1}{F_q(\kappa)}\right)^{\frac{\alpha}{d+\alpha}}	 \inf_{ \widehat{U}\subset \mathcal U_m}\left[\lambda_1^m(\widehat U) +
	(1-\zeta) |\widehat U|\right].
	\end{eqnarray*}
	Consequently, as long as $t$ is large enough,
	\begin{eqnarray*}
L(t) &\leq& C_1 2^{2(\frac{\sqrt d m}{b})^d ((t-1)/-\ln F_q(\kappa))^{\frac{d}{d+\alpha}}}\\
 && \times\exp\left\{ -\left( (t-1)^{\frac{d}{d+\alpha}} \left(\ln\frac{1}{F_q(\kappa)}\right)^{\frac{\alpha}{d+\alpha}} \inf_{U\in\mathcal U_m} [\lambda_1^m(U)+(1-\zeta) |U|]\right) \wedge A\right.\\
  &&\left. \phantom{paaaaaaaa}+\delta \left({(t-1)^{\frac{d}{d+\alpha}}}\left(\ln\frac{1}{F_q(\kappa)}\right)^{\frac{\alpha}{d+\alpha}} +m^d\ln\frac{1}{F_q(\kappa)} \right)\right\}. 	
	\end{eqnarray*}
It follows
\[\limsup_{t\to\infty} \frac{\ln L(t)}{t^{\frac{d}{d+\alpha}}\left(\ln\frac{1}{F_q(\kappa)}\right)^{\frac{\alpha}{d+\alpha}}}\leq 2(\frac{m\sqrt d}{b})^d\ln 2 - \inf_{U\in\mathcal U_m} [\lambda_1^m(U)+(1-\zeta) |U|] \wedge A +{\delta}.
\]
Now we let $b\to\infty, A\to\infty, \delta\to 0,$ which results in
\[\limsup_{t\to\infty} \frac{\ln L(t)}{t^{\frac{d}{d+\alpha}}\left(\ln\frac{1}{F_q(\kappa)}\right)^{\frac{\alpha}{d+\alpha}}}\leq -\inf_{U\in\mathcal U_m} [\lambda_1^m(U)+(1-\zeta) |U|],
\]
for any $m>0,$ $\zeta\in(0,1).$
Therefore, by Lemma \ref{lem:sup},
\begin{eqnarray*}
\limsup_{t\to\infty} \frac{\ln L(t)}{t^{\frac{d}{d+\alpha}}\left(\ln\frac{1}{F_q(\kappa)}\right)^{\frac{\alpha}{d+\alpha}}}& \leq&  -\sup_m\inf_{U\in\mathcal U_m} \left[\lambda_1^{m}(U)+(1-\zeta) |U|\right] \\
&\leq&
-\inf_{U\in\mathcal U} \left[\lambda_1(U)+(1-\zeta) |U|\right]
=- (1-\zeta)^{\frac{\alpha}{d+\alpha}} C_{d,\alpha},
\end{eqnarray*}
where $\mathcal U$ is the collection of all open subsets of $\R^d.$
	As the last step, we pass to the limit $\zeta\to 0$ and we obtain \eqref{eq:upper-statement}.
	The proof is complete.
\hfill$\Box$

\subsection{Formula for the rate function}
In the concluding lines of the argument above we have used the following lemma, that can be tracked down to the Donsker-Varadhan paper \cite{bib:Don-Var}.
See also \cite[Lemma 3.3]{bib:Szn1} for its Brownian motion counterpart.

\begin{lemma}\label{lem:sup}
Fix $\alpha\in(0,2]$. Let $\lambda_1(U)$ (resp. $\lambda_1^m(U),$ $m=1,2,...$) be the principal eigenvalue of the operator $(-\Delta)^{\alpha/2}$  (resp. the operator $-A^m$, where $-A^m$ is the generator of the stable semigroup on $\mathcal T_m$) constrained to an open set $U \subset \R^d$ (resp. $U\subset \mathcal T_m),$  with Dirichlet conditions on $U^c$ when $\alpha\in(0,2)$ and on $\partial U$ when $\alpha=2$ (cf. \eqref{eq:def=lambdaMU}). Denote by $\mathcal U$ (resp. $\mathcal U_m$) the collection
of all open subsets of $\R^d$ (resp. $\mathcal T_m$). Let $\nu>0$ be a given number.
Then
\begin{equation}\label{eq:sup}
\sup_m\inf_{U\in \mathcal U_m}[\lambda_1^{m}(U)+\nu |U|] \geq \inf_{U\in \mathcal U}[\lambda_1(U)+\nu |U|]= C_{d,\alpha}\nu^{\frac{\alpha}{d+\alpha}},
\end{equation}
with $C_{d,\alpha}$ given by \eqref{eq:opt-constant}.
\end{lemma}

\begin{proof}
Case $\alpha=2$ is covered by \cite[Lemma 3.3]{bib:Szn1}, so that we assume $\alpha<2.$  We combine the lines of the proofs of  this lemma and of \cite[Lemma 3.5]{bib:Don-Var}.
We will show that for any $\delta\in(0,1),$ there is an $m$ such that for any $U\in\mathcal U_m$ there exists $V\in\mathcal U$ such that
\begin{equation}\label{eq:app1} {\lambda_1(V)+\nu|V|} \leq (1+\delta) \left(\lambda_1^{m} (U)+\nu|U|\right)+\delta.
\end{equation}
This will do, as then
\begin{eqnarray*}
	\inf_{V\in\mathcal U} [{\lambda_1(V)+\nu|V|}]& \leq& (1+\delta) \left[\inf_{U\in\mathcal U_m}[ \lambda_1^m (U)+\nu|U|]\right]+\delta\\
	&\leq & (1+\delta) \sup_m\left[\inf_{U\in\mathcal U_m}[ \lambda_1^m (U)+\nu|U|]\right]+\delta,
\end{eqnarray*}
and passing to the limit $\delta\to 0$ gives the statement.

To determine the actual value of the constant on the right-hand side of \eqref{eq:sup}, one first notices that for the isotropic $\alpha-$stable process it is enough to consider balls, not arbitrary open subsets of $\R^d$
(this follows from the Faber-Krahn inequality, see e.g. \cite[Lemma 3.13]{bib:Don-Var}), use scaling of the principal eigenvalue, and minimize over the radius of balls.

To prove \eqref{eq:app1},
we will use the variational expression for the principal Dirichlet eigenvalue given by  \eqref{eq:def=lambdaMU}.
For $m$ and $U\in\mathcal U_m,$ fixed for the moment,
pick $\varphi$  in \eqref{eq:def=lambdaMU} for which \[\mathcal E^{m}(\varphi,\varphi) =  \lambda_1^{m} (U)\]
(the ground state eigenfunction).

 Since the torus $\mathcal T_m$ is identified with a box $[0,m)^d$, it makes sense to consider the counterpart of $\varphi$ defined on this box and then extended periodically to the whole $\R^d$. Throughout we denote this extension by $\widetilde \varphi$.
As in \cite{bib:Szn1} we can find a point $y_0\in [0,m)^d$ for  which
\begin{align} \label{eq:y0_est}
  \int_{[0,m)^d} \widetilde\varphi^2(x) \mathbf 1_{[\sqrt m, m-\sqrt m]^d}(x-y_0) \,{\rm d}x \geq 1-\frac{2d}{\sqrt m}.
	\end{align}
Indeed, this assertion follows from the estimate
\begin{align*}
\frac{1}{m^d} \int_{[0,m)^d} & \left(\int_{[0,m)^d}\widetilde\varphi^2(x)\mathbf 1_{[0,m)^d\setminus [\sqrt m, m-\sqrt m]^d}(x-y)\,{\rm d}x\right){\rm d}y \\ & = \frac{|[0,m)^d\setminus [\sqrt m,m-\sqrt m]^d|}{m^d}\int_{[0,m)^d}\widetilde\varphi^2(x)\,{\rm d}x\leq\frac{2d}{\sqrt m}.
\end{align*} \normal
Let now $V=\pi_m^{-1}(U) \cap([0,m)^d)+y_0) \subset \R^d.$ We have  $|V|=|U|.$
Introduce the function $h_m:\mathbb R\to [0,1]$ that is equal to $0$ for $x<0$ or $x>m,$ equal to 1 for $x\in(\sqrt m, m-\sqrt m),$ increases from 0 to 1 on $(0,\sqrt m),$ decreases from 1 to 0 on $(m-\sqrt m,m)$, is of class $C^1,$ and $\sup_{t\in \mathbb R} |h_m'(t)|\leq \frac{2}{\sqrt m}.$  Then for $x\in \R^d$ let $$H_m(x)= h_m(x_1)\cdots h_m(x_d).$$
Observe that
\begin{equation}\label{eq:grad-of-Hm}
\rho_m:=\sup_{x\in\R^d} \|\nabla H_m(x)\| \leq 2\sqrt{\frac{d}{m}}.
\end{equation}
 Define
$\psi(x) = \widetilde \varphi(x)H_m(x-y_0).$  This function is supported in $V$ and by \eqref{eq:y0_est} it satisfies
$$
1\geq  \int_{\R^d}  \psi^2(x) {\rm d}x \geq \int_{[0,m)^d} \widetilde\varphi^2(x) \mathbf 1_{[\sqrt m, m-\sqrt m]^d}(x-y_0) \,{\rm d}x \geq 1-\frac{2d}{\sqrt m}.
$$

We now evaluate $\mathcal E(\psi,\psi),$ i.e. the value of the Dirichlet form of the $\alpha-$stable process on $\R^d$ for the function $\psi.$  We do not check separately that $\psi\in \mathcal D(\mathcal E)$ (the domain of the Dirichlet form of the $\alpha-$stable process on $\R^d),$ it will become clear from the  arguments that follow.
From the symmetry of $\nu$ we have:
\begin{eqnarray*}
	\mathcal E(\psi,\psi) & = & \int_{\R^d} \int_{\R^d} (\psi(x)-\psi(z))^2\nu(x,z)\,{\rm d}x{\rm d}z\\
	&= & \int_{[0,m)^d+y_0} \int_{[0,m)^d+y_0} (\psi(x)-\psi(z))^2\nu(x,z)\,{\rm d}x{\rm d}z\\
 && +2 \int_{[0,m)^d+y_0} \int_{([0,m)^d+y_0)^c} \psi(z)^2\nu(x,z){\rm d}x{\rm d}z =: A_m[\psi] +2B_m[\psi].
\end{eqnarray*}	
Using the Minkowski inequality we get
\begin{eqnarray}\label{eq:est-a}
	A_m[\psi]&=&
	\int_{[0,m)^d+y_0} \int_{[0,m)^d+y_0} (\psi(x)-\psi(z))^2\nu(x,z)\,{\rm d}x{\rm d}z
	\nonumber\\
	&\leq & \left(\sqrt{\int_{[0,m)^d+y_0}\int_{[0,m)^d+y_0} H_m(x-y_0)^2 (\widetilde\varphi(x)-\widetilde\varphi(z))^2 \nu(x,z) {\rm d}x {\rm d}z}\right.\nonumber\\
	&&\left. + \sqrt{   \int_{[0,m)^d+y_0}\int_{[0,m)^d+y_0} \widetilde\varphi^2(z)(H_m(x-y_0)-H_m(z-y_0))^2 \nu(x,z) {\rm d}x {\rm d}z}\right)^{2}
\end{eqnarray}	
Since $|H_m(x)|\leq 1 $  for $x\in[0,m)^d$ and $H_m(x)=0$ for $x$ outside of this box, the first of the integrals  is bounded from the above by
\begin{eqnarray*}
&&\int_{[0,m)^d+y_0}\int_{\mathbb R^d}(\widetilde\varphi(x)-\widetilde\varphi(z))^2\nu(x,z){\rm d}z{\rm d}x = \int_{[0,m)^d}\int_{\mathbb R^d}(\widetilde\varphi(x)-\widetilde\varphi(z))^2\nu(x,z){\rm d}z{\rm d}x.
\end{eqnarray*}
The latter equality is a consequence of the periodicity of $\widetilde \varphi$ and the translation invariance on $\nu$. Further, by using the definition of $\nu_m$ and the periodicity of $\widetilde \varphi$ again, the last integral above is equal to
\begin{eqnarray*}
&&\int_{[0,m)^d} \sum_{\mathbf i\in\mathbb Z^d}\int_{[0,d)^m+m\mathbf i}(\widetilde\varphi(x)-\widetilde\varphi(z))^2\nu(x,z){\rm d}z{\rm d}x \\
&=& \int_{[0,m)^d}\int_{[0,m)^d}
(\widetilde\varphi(x)-\widetilde\varphi(z))\sum_{\mathbf i\in\mathbb Z^d} \nu(x, z+m\mathbf i){\rm d}z{\rm d}x
\\
&=&
\int_{[0,m)^d}\int_{[0,m)^d}  (\widetilde\varphi(x)-\widetilde\varphi(z))^2\nu_m(x,z){\rm d}z {\rm d}x\\
&=& \int_{\mathcal T_m}\int_{\mathcal T_m}(\varphi(x)-\varphi(z))^2 \nu_m(x,z){\rm d}x{\rm d}z = \mathcal E^{(m)}(\varphi,\varphi)  .
\end{eqnarray*}

We estimate  the other integral in \eqref{eq:est-a}  by
\begin{eqnarray} \label{eq:sup-1}
&&\int_{[0,m)^d} \widetilde\varphi^2(z'+y) \left(\sup_{z'\in [0,m)^d}\int_{[0,m)^d} (H_m(x')-H_m(z'))^2\nu(x',z'){\rm d}x'\right){\rm d}z' \nonumber \\
&\leq& \sup_{z'\in \R^d}\int_{\R^d} (H_m(x')-H_m(z'))^2\nu(x',z'){\rm d}x'  =: s_m,
\end{eqnarray}
because  $\int_{[0,m)^d} \widetilde \varphi^2(x){\rm d}x =1$.
These estimates add up to
\[
\mathcal A_m[\psi] \leq \left(\sqrt{\mathcal E^{(m)}(\varphi,\varphi)} + \sqrt{s_m}\right)^2.
\]
We now estimate $B_m[\psi].$ Again, since $H_m$ is supported in $(0,m)^d,$ we can write:
\begin{eqnarray*}
B_m[\psi]&=& \int_{[0,m)^d+y_0} \int_{([0,m)^d+y_0)^c} H_m^2(z-y_0)\widetilde\varphi(z)^2\nu(x,z){\rm d}x{\rm d}z\\
&=& \int_{[0,m)^d}\int_{([0,m)^d)^c} H_m^2(z)\widetilde\varphi(z+y_0)^2 \nu(x,z)\,{\rm d}x{\rm d}z\\
&=& \int_{[0,m)^d}\int_{([0,m)^d)^c} (H_m(z)-H_m(x))^2\widetilde\varphi(z+y_0)^2 \nu(x,z)\,{\rm d}x{\rm d}z\\
&\leq& \int_{[0,m)^d}\widetilde\varphi^2(z+y_0)\int_{\R^d} (H_m(z)-H_m(x))^2\nu(x,z){\rm d}x {\rm d}z \leq s_m.
\end{eqnarray*}
Altogether,
$$\mathcal E(\psi,\psi) \leq   \left(\sqrt{\mathcal E^{(m)}(\varphi,\varphi)} + \sqrt{s_m}\right)^2 +2s_m.$$

The supremum $s_m$ can be estimated as follows.  By the fact that $\nu(x,z) = \cA_{d,-\alpha}|z-x|^{-d-\alpha}$, $x \neq z$,  for any $N>1$ we have
\begin{eqnarray*} &&\sup_{x\in\R^d}	\int_{\R^d} (H_m(x)-H_m(z))^2\nu(x,z){\rm d}x
	\\
	&\leq& \sup_{x\in\R^d}
	\left(\int_{|z-x|\leq N} |H_m(x)-H_m(z)|^2 \nu(x,z){\rm d}z + \int_{|z-y|\geq N} \nu(x,z){\rm d}z\right)\\
	&\leq &  \sup\|\nabla H_m(x)\|^2 \cA_{d,-\alpha} \int_{|z|\leq N} |z|^{-d-\alpha+2}{\rm d}z + \cA_{d,-\alpha} \int_{|z|>N} |z|^{-d-\alpha}{\rm d}z \\
&\leq &  \rho_m^2 \cA_{d,-\alpha} N^{2-\alpha} +  \cA_{d,-\alpha} N^{-\alpha} .
\end{eqnarray*}
Therefore, if $m$ is sufficiently large, $s_m$ can be made arbitrarily small.
It follows:
\begin{eqnarray*}
	\lambda_1(V)&\leq&  \mathcal E\left(\frac{\psi}{\|\psi\|_2},\frac{\psi}{\|\psi\|_2}\right) \leq \left(\frac{\sqrt m}{\sqrt m-2d}\right)^2  \left(\sqrt {\lambda_1^m(U)} +\sqrt{s_m}\right)^2 +2s_m.
\end{eqnarray*}

Using the inequality
\[(a+b)^2\leq (1+\eta^2)a^2+ (1+\eta^{-2})b^2, \quad a,b\in\mathbb R, \ \  \eta>0\]
we further have, for any $\eta>0,$ and $m>m_0$ $(m_0$ is so large that $(1+\frac{2d}{\sqrt m -2d})^2\leq 2$)
\begin{eqnarray*}
\lambda_1(V)&\leq & \left(1+\frac{2d}{\sqrt m-2d}\right)^2(1+\eta^2)\lambda^m_1(U)+\left(4+\eta^{-2}\right)s_m
\end{eqnarray*}
For a given $\delta\in(0,1),$ we now take $\eta=\sqrt \delta/2,$  and $m$ so large that simultaneously\linebreak $(1+\frac{2d}{\sqrt m -2d})^2(1+\frac{\delta}{4})\leq(1+\delta)$ and $(4+\frac{4}{\delta^2})s_m\leq\delta.$
For  $m$ this big, we have
\[{\lambda_1(V)} \leq (1+\delta)   {\lambda_1 ^{m}(U)}+\delta,\]
and since $|U|= |V|,$ this gives
\[ {\lambda_1(V)+\nu|V|} \leq (1+\delta) \left(\lambda_1^{m} (U)+\nu|U|\right)+\delta.\]
\normal
\end{proof}

\section{Proof of Theorem \ref{th:dolne}}\label{sec:lower-proof}

 For a ball $B_r,$    let
\[\mathcal A_0 := \{\omega: q_i(\omega)=0 \mbox{ for all } i\in B_r^{2a}\},\]
where $a$ is the range of the  profile $W$  and $B_r^\rho:= \{x\in\R^d: \mbox{dist}\,(x,B_r)<\rho\}= B_{r+\rho}.$ As the random variables $q_{\mathbf i}$ are independent, the probability of this event is given by:
\[\mathbb Q[\mathcal A_0]=(\mathbb Q[q=0])^{\#\{i\in B^{2a}_r\}}
= {\rm e}^{-(1+\epsilon_r)|B_r|\ln(\frac{1}{\mathbb Q[q=0]})},\]
with $\epsilon_r=o(1)$ as $r\to\infty.$

By Corollary \ref{coro-el-t-properties} (ii), we have
\begin{eqnarray*}
	L(t) &=& (1+o(1))\frac{1}{|B_r|}\int_{B_r} p(t,x,x) \mathbb E^{\mathbb Q} \mathbf E_{x,x}^t\left[{\rm e}^{-\int_0^t V^\omega(Z_s){\rm d}s}   \right]\,{\rm d}x
\\
&\geq &  (1+o(1))	\frac{1}{|B_r|}\int_{B_r} p(t,x,x) \mathbb E^{\mathbb Q} \mathbf E_{x,x}^t\left[{\rm e}^{-\int_0^t V^\omega(Z_s){\rm d}s}; A_0, \tau_{B_r}>t   \right]\,{\rm d}x.
\end{eqnarray*}
On the event $A_0\cap\{\tau_{B_r}>t\}$ we have that
\(
\int_0^t V^\omega(Z_s)\,{\rm d}s=0,\) so that
our estimate continues as
\begin{eqnarray}\label{eq:up-to-here}
L(t)	&\geq&(1+o(1))\mathbb Q[\mathcal A_0]\cdot \frac{1}{|B_r|}\int_{B_r} p (t,x,x) \mathbf P^t_{x,x}[\tau_{B_r}>t]{\rm d}x \nonumber \\
	&=& \frac{(1+o(1))}{|B_r|}\,{\rm e}^{ - (1+\epsilon_r)|B_r|\ln\frac{1}{\mathbb Q[q=0]}} \mbox{Tr}\, T_t^{B_r} \nonumber\\
	&\geq& \frac{(1+o(1))}{|B_r|}\,{\rm e}^{- (1+\epsilon_r)|B_r|\ln\frac{1}{\mathbb Q[q=0]}} {\rm e} ^{-t \lambda_1 (B_r)},\quad r\to\infty.
\end{eqnarray}
Moreover, for any ball $B_r$ we have the scaling
\begin{equation}\label{eq:scal}
|B_r|=r^d\omega_d \quad \mbox{ and } \qquad   \lambda_1(B_r)= \frac{1}{r^\alpha}\lambda_d^{(\alpha)},
\end{equation}
where $\omega_d$ is the volume of the   unit ball in $\R^d$ and $\lambda_d^{(\alpha)}$ is the corresponding  principal eigenvalue, so that for sufficiently large $r$ we have
\[
L(t)\geq \frac{(1+o(1))}{r^d\omega_d} {\rm e}^{-\left[(1+\epsilon_r)(\omega_d r^d\ln \frac{1}{\mathbb Q[q=0]})+ \frac{t}{r^\alpha}\,\lambda_d^{(\alpha)}
\right]}
\]
We choose   \[r=r_t=\left(\frac{\alpha \lambda_d^{(\alpha)}}{\ln\frac{1}{\mathbb Q[q=0]}d\omega_d}\,t\right)^\frac{1}{d+\alpha}.\]
Clearly,  we have $r_t\to\infty$ when $t\to\infty.$

With this substitution we obtain
\[L(t)\geq \frac{(1+o(t))}{a(t)} {\rm e}^{ -t^{\frac{d}{d+\alpha}}\left(\ln \frac{1}{\mathbb Q[q=0]}\right)^{\frac{\alpha}{d+\alpha}}
\omega_d^{\frac{\alpha}{d+\alpha}} (\lambda_d^{(\alpha)})^{\frac{d}{d+\alpha}}
\left(
(1+\epsilon_{r_t})(\alpha/d)^{\frac{d}{d+\alpha}}+(d/\alpha)^{\frac{\alpha}{d+\alpha}}\right)}
\]
($a(t)$ is polynomial in $t,$ the exact formula is not necessary),
and further, taking into account the property that $\epsilon_{r_t}\to 0,$
\begin{eqnarray*}
\liminf_{t\to\infty}
\frac{\log L(t)}{t^{\frac{d}{d+\alpha}}\left(\ln\frac{1}{\mathbb Q[q=0]}\right)^{\frac{\alpha}{d+\alpha}}}
&\geq& - \omega_d^{\frac{\alpha}{d+\alpha}}
(\lambda_d^{(\alpha)})^{\frac{d}{d+\alpha}}\left((\alpha/d)^{\frac{d}{d+\alpha}} +(d/\alpha)^{\frac{\alpha}{d+\alpha}}\right)\\
&=& -\omega_d^{\frac{\alpha}{d+\alpha}}\left(\frac{d+\alpha}{\alpha}\right)
\left(\frac{\alpha \lambda_d^{(\alpha)}}{d}\right)^{\frac{d}{d+\alpha}} = -C_{d,\alpha}.
\end{eqnarray*}
This concludes the proof.

\iffalse
\section{Concluding remarks}
We would like to point out that our arguments can be applied to
more general L\'{e}vy processes -- those that can be compared with
$\alpha-$stable processes. The Poissonian-type Lifschitz tail will be
present when $\mathbb Q[q=0]>0.$ The `existence' part of our proofs remains unchanged, and as to the asymptotics, we just need to replace the exact scaling of the principal eigenvalue of $\alpha-$stable processes needed in formulas
\eqref{eq:lambda-scal} and \eqref{eq:scal} with one-sided expressions:
\(\lambda_1^{V_{Km}^{*,\kappa}(\mathcal T_{Km})}\geq \frac{const}{K^\alpha}\lambda_1^{\widetilde{V}^{*,\kappa,(\alpha)}_{K,m}}(\mathcal T_m)\) (in \eqref{eq:lambda-scal}) and $
\lambda_1(B_r)\leq\frac{const}{r^\alpha}\lambda_d^{(\alpha)}$ (in \eqref{eq:scal}), and choose $m=m_0$ (even $m_0=1$) in the upper bound. We  will get the upper and lower bounds with rate functions identical as before, but with different upper and lower constants. The result in this case will be:

\noindent when $\mathbb Q[q=0]=0$, then
\[\lim_{\lambda\to 0} \lambda^{d/\alpha}\ln\ell([0,\lambda])=-\infty,\]
and when $\mathbb Q[q=0]>0,$ then there exist two positive constants $C,D>0$ for which
\[-C\leq\liminf_{\lambda\to 0} \lambda^{d/\alpha}\ln\ell([0,\lambda])\leq\limsup_{\lambda\to 0}\lambda^{d/\alpha}\ln\ell([0,\lambda]) \leq -D.\]
\fi

\end{document}